\documentclass{article}
\pdfoutput=1

\usepackage{my-preamble}
\usepackage[style=authoryear-comp]{biblatex}

\DeclareNameAlias{author}{given-family}

\title{\TITLE}
\author{\AUTHOR}

\begin{document}

\maketitle

\begin{abstract}
  We show that the essentially algebraic theory of generalized
algebraic theories, regarded as a category with finite limits, has a
universal exponentiable arrow in the sense that any exponentiable
arrow in any category with finite limits is the image of the
universal exponentiable arrow by an essentially unique functor.


\end{abstract}

\section{Introduction}
\label{sec:introduction}

An arrow in a category with finite limits is said to be
\emph{exponentiable} if the pullback functor along the arrow has a
right adjoint. In this paper we construct a \emph{universal}
exponentiable arrow in the sense that any exponentiable arrow in any
category with finite limits is the image of the universal
exponentiable arrow by an essentially unique functor.

The universal exponentiable arrow comes from the theory of
\emph{generalized algebraic theories}
\parencite{cartmell1978generalised} which are equational theories
written in the dependent type theory without any type
constructors. Let \(\thgat\) be the opposite of the category of finite
generalized algebraic theories and (equivalence classes of)
interpretations between them. The category \(\thgat\) has finite
limits. Our main result is that the category \(\thgat\) has a
universal exponentiable arrow (\cref{thm:gat-ump}).

For an exponentiable arrow \(\pr : \tot \to \bas\) in a category
\(\cart\) with finite limits, the \emph{associated polynomial functor}
\(\poly_{\pr} : \cart \to \cart\)
\parencite{gambino2013polynomial,weber2015polynomials} is of much
interest. For an object \(\obj \in \cart\), the object
\(\poly_{\pr}\obj\) is known to satisfy the universal property of the
\emph{partial product} of \(\obj\) over \(\pr\), and conversely, if
all partial products over \(\pr\) exist then \(\pr\) is exponentiable
\parencite{dyckhoff1987exponentiable,niefield1982cartesianness}. In
\cref{sec:exponentiable-arrows} we give an algebraic characterization
of the polynomial functor \(\poly_{\pr}\). Precisely, we show that if
an endofunctor \(\poly : \cart \to \cart\) is equipped with certain
natural transformations, then \(\poly\obj\) for an object
\(\obj \in \cart\) is the partial product of \(\obj\) over \(\pr\),
and thus \(\pr\) is exponentiable and \(\poly\) must be isomorphic to
the associated polynomial functor \(\poly_{\pr}\).

Using results in \cref{sec:exponentiable-arrows}, we construct an
exponentiable arrow in \(\thgat\) in \cref{sec:depend-type-theor}. In
fact, we find an exponentiable arrow for any type theory satisfying
certain mild assumptions. We view type theories as logical frameworks
\parencite{harper1993framework,nordstrom2001martin-lof}, which are
frameworks for defining \emph{theories}, and for each type theory
\(\tth\), we obtain a category of \emph{\(\tth\)-theories} and
interpretations between them. Let \(\thdtt_{\tth}\) be the opposite of
the category of finite \(\tth\)-theories. The category
\(\thdtt_{\tth}\) always has a special arrow
\(\typeof_{0} : \El_{0} \to \Ty_{0}\), where \(\Ty_{0}\) is the theory
generated by a constant type and \(\El_{0}\) is the theory generated
by a constant type and a constant term of the type. We show that the
structural rules of \emph{weakening}, \emph{projection} and
\emph{substitution} yield a polynomial functor
\(\poly_{\typeof_{0}} : \thdtt_{\tth} \to \thdtt_{\tth}\) for
\(\typeof_{0}\) (\cref{prop:dtt-family-structure}). Consequently,
\(\typeof_{0}\) is exponentiable. When \(\tth\) is the dependent type
theory without any type constructors, we show in
\cref{sec:univ-expon-arrow} that the exponentiable arrow in
\(\thdtt_{\tth} = \thgat\) is, moreover, a universal exponentiable
arrow.

\subsection{Related Work}
\label{sec:related-work}

This work was started as part of a categorical approach to a general
notion of a type theory given by the author
\parencite{uemura2019framework}, but it turned out that the
construction of the universal exponentiable arrow is interesting in
itself, so the author decided to write a separate paper. In that paper
\parencite{uemura2019framework}, the author explained from a semantic
point of view that a type theory can be identified with a category
equipped with a class of exponentiable arrows. This paper provides a
syntactic justification for this definition of a type theory:
exponentiable arrows naturally appear in categories of theories.

We call the category \(\thgat\) the \emph{essentially algebraic theory
  of generalized algebraic theories} because, by the Gabriel-Ulmer
duality \parencite{gabriel1971lokal}, the category of generalized
algebraic theories is equivalent to the category of ``models of
\(\thgat\)'', that is, functors \(\thgat \to \Set\) preserving finite
limits. The view of theories as categories originates from Lawvere's
functorial semantics of algebraic theories
\parencite{lawvere2004functorial}. There are several descriptions of
the essentially algebraic theory of generalized algebraic
theories. \Textcite{cartmell1978generalised} showed that it is
equivalent to the essentially algebraic theory of contextual
categories. \Textcite{isaev2017algebraic,voevodsky2014b-system}
proposed alternative essentially algebraic theories which have sorts
of types, sorts of terms and operator symbols for weakening,
projection and substitution. \Textcite{garner2015combinatorial}
constructed a monad on a presheaf category whose algebras are the
generalized algebraic theories. Our contribution is to give a simple
universal property of the essentially algebraic theory of generalized
algebraic theories: it is the initial essentially algebraic theory
with an exponentiable arrow.

\Textcite{fiore2010second,fiore2014functorial} used a universal exponential object to
give functorial semantics of algebraic theories in languages with
variable binding. Our universal exponentiable arrow provides a form of
functorial semantics of type theories. Indeed, a natural model of type
theory \parencite{awodey2018natural,fiore2012discrete} can be regarded
as a functor from \(\thgat\) to a presheaf category preserving finite
limits and pushforwards along \(\typeof_{0}\). See
\parencite{uemura2019framework} for details of this style of
functorial semantics of type theories.

Polynomial functors are extensively studied in a wide range of areas
of mathematics and computer science. We refer the reader to
\parencite{gambino2013polynomial} for general information. Our
characterization of the polynomial functor \(\poly_{\pr}\) associated
to \(\pr : \tot \to \bas\) given in \cref{sec:exponentiable-arrows} is
based on the equivalence of the exponentiability of \(\pr\) and the
existence of a right adjoint of \((\argu \times_{\bas} \tot)\)
\parencite{niefield1982cartesianness}. The class of polynomial
functors is known to be characterized as the class of local fibred
right adjoints \parencite{kock2013local}. One application of
polynomial functors to the study of dependent type theory is the
semantics of inductive types
\parencite{moerdijk2000wellfounded,gambino2004wellfounded,abbott2005containers}. The
polynomial functor associated to the universal exponentiable arrow
\(\typeof_{0}\) is related to the use of polynomial functors for
modeling type constructors on natural models
\parencite{awodey2018natural,newstead2018thesis}; see
\cref{rem:natural-model}.

Exponentiable morphisms have been studied especially in categories of
spaces
\parencite{niefield1982cartesianness,niefield2001exponentiable}.
Exponentiability in categories of theories has received less
attention, but some exponentiable morphisms of theories are known. For
example, classifying toposes of coherent theories over a topos
\(\topos\) are exponentiable in the (\(2\)-)category of bounded
\(\topos\)-toposes and geometric morphisms over \(\topos\)
\parencite{johnstone2002sketches}.

\section{Exponentiable Arrows and Polynomial Functors}
\label{sec:exponentiable-arrows}

In this section we recall the definition of an exponentiable arrow and
show that an arrow is exponentiable if and only if there exists a
polynomial functor for it (\cref{thm:family-repr}).

\begin{definition}
  A \emph{cartesian category} is a category that has finite limits. A
  \emph{cartesian functor} between cartesian categories is a functor
  that preserves finite limits. For cartesian categories \(\cart\) and
  \(\cartI\), we denote by \(\Cart(\cart, \cartI)\) the category of
  cartesian functors \(\cart \to \cartI\) and natural
  transformations between them.
\end{definition}

\begin{notation}
  Let \(\cart\) be a cartesian category. For an arrow
  \(\pr : \tot \to \bas\) in \(\cart\), we denote by
  \(\pr^{*} : \cart/\bas \to \cart/\tot\) the pullback functor along
  \(\pr\) and by \(\pr_{!} : \cart/\tot \to \cart/\bas\) its left
  adjoint, that is, the postcomposition with \(\pr\). For an object
  \(\obj \in \cart\), we denote by \(\obj^{*} : \cart \to \cart/\obj\)
  the pullback along the arrow \(\obj \to \term\) to the terminal
  object and by \(\obj_{!} : \cart/\obj \to \cart\) its left adjoint.
\end{notation}

Any object in a slice category \(\cart/\bas\) is written as a pullback
of objects of the form \(\bas^{*}\obj \cong \obj \times \bas\) with
\(\obj \in \cart\) as follows.

\begin{lemma}
  \label{lem:object-in-slice}
  Let \(\cart\) be a cartesian category, \(\bas \in \cart\) an
  object. For any object \((\arr : \obj \to \bas) \in \cart/\bas\), we
  have the following pullback in \(\cart/\bas\)
  \[
    \begin{tikzcd}
      \obj
      \arrow[r, "{(\id_{\obj}, \arr)}"]
      \arrow[d, "\arr"'] &
      \obj \times \bas
      \arrow[d, "\arr \times \bas"] \\
      \bas
      \arrow[r, "\diagonal_{\bas}"'] &
      \bas \times \bas,
    \end{tikzcd}
  \]
  where \(\diagonal_{\bas}\) is the diagonal arrow. \qed
\end{lemma}

\begin{definition}
  Let \(\pr : \tot \to \bas\) be an arrow in a cartesian category
  \(\cart\). We say \(\pr\) is \emph{exponentiable} if the pullback
  functor \(\pr^{*} : \cart/\bas \to \cart/\tot\) has a right
  adjoint. When \(\pr\) is exponentiable, the right adjoint of
  \(\pr^{*}\) is called the \emph{pushforward along \(\pr\)} and denoted
  by \(\pr_{*}\).
\end{definition}

\begin{construction}
  Let \(\pr : \tot \to \bas\) be an exponentiable arrow in a cartesian
  category \(\cart\). We define the \emph{polynomial functor
    \(\poly_{\pr} : \cart \to \cart\) associated to \(\pr\)} to be the
  composite
  \[
    \begin{tikzcd}
      \cart
      \arrow[r,"\tot^{*}"] &
      \cart/\tot
      \arrow[r,"\pr_{*}"] &
      \cart/\bas
      \arrow[r,"\bas_{!}"] &
      \cart.
    \end{tikzcd}
  \]
\end{construction}

We characterize the associated polynomial functor \(\poly_{\pr}\) as
an endofunctor on \(\cart\) equipped with certain natural
transformations and show that the existence of a polynomial functor
for \(\pr\) is equivalent to the exponentiability of \(\pr\). First,
\(\poly_{\pr}\obj\) for an object \(\obj \in \cart\) is equipped with
an arrow \(\leg_{\obj} : \poly_{\pr}\obj \to \bas\), which defines a
natural transformation \(\leg : \poly_{\pr} \To \bas\). Then
\(\poly_{\pr} : \cart \to \cart\) and \(\leg : \poly_{\pr} \To \bas\)
determines a functor \(\cart \to \cart/\bas\) which is nothing but
\(\pr_{*}\tot^{*}\). Since \(\pr_{*}\tot^{*}\) is the right adjoint of
\(\tot_{!}\pr^{*} \cong (\argu \times_{\bas} \tot)\), we have the unit
\(\unit_{\objI} : \objI \to \poly_{\pr}(\objI \times_{\bas} \tot)\)
for \(\objI \in \cart/\bas\) and the counit
\(\subst_{\obj} : \poly_{\pr}\obj \times_{\bas} \tot \to \obj\) for
\(\obj \in \cart\) satisfying the following triangle identities.
\begin{align}
  \label{eq:15}
  \poly_{\pr}\subst_{\obj} \circ \unit_{\poly_{\pr}\obj}
  &= \id_{\poly_{\pr}\obj}
  & (\obj \in \cart) \\
  \label{eq:16}
  \subst_{\objI \times_{\bas} \tot} \circ (\unit_{\objI} \times_{\bas}
  \tot)
  &= \id_{\objI \times_{\bas} \tot}
  & (\objI \in \cart/\bas)
\end{align}
This characterizes the polynomial functor as the \emph{partial product
  over \(\pr\)} \parencite{dyckhoff1987exponentiable}, and it is known
that the existence of such an endofunctor
\(\poly_{\pr} : \cart \to \cart\) with natural transformations
\(\leg\), \(\unit\) and \(\subst\) is equivalent to the
exponentiability of \(\pr\).
\begin{proposition}[\textcite{niefield1982cartesianness}]
  \label{prop:exp-poly}
  For an arrow \(\pr : \tot \to \bas\) in a cartesian category
  \(\cart\), the following are equivalent:
  \begin{enumerate}
  \item \(\pr\) is exponentiable;
  \item \((\argu \times_{\bas} \tot) : \cart/\bas \to \cart\) has a
    right adjoint.
  \end{enumerate}
\end{proposition}

We further decompose the unit \(\unit_{\objI}\) into pieces. As
special components of \(\unit\), we have the following arrows over
\(\bas\).
\begin{align*}
  \unit_{\obj \times \bas}
  &: \obj \times \bas \to \poly_{\pr}((\obj \times \bas) \times_{\bas}
    \tot) \cong \poly_{\pr}(\obj \times \tot)
  & (\obj \in \cart) \\
  \unit_{\bas}
  &: \bas \to \poly_{\pr}(\bas \times_{\bas} \tot) \cong
    \poly_{\pr}\tot
\end{align*}
We define a natural transformation
\[
  \wk_{\obj} : \obj \times \bas \to \poly_{\pr}\obj
\]
for \(\obj \in \cart\) by the composite
\[
  \begin{tikzcd}
    \obj \times \bas
    \arrow[r, "\unit_{\obj \times \bas}"] &
    \poly_{\pr}(\obj \times \tot)
    \arrow[r, "\poly_{\pr}\prodpr_{1}"] &
    \poly_{\pr}\obj,
  \end{tikzcd}
\]
where \(\prodpr_{1} : \obj \times \tot \to \obj\) is the first
projection, and let
\[
  \proj := \unit_{\bas} : \bas \to \poly_{\pr}\tot.
\]
The unit \(\unit\) can be recovered from \(\wk\) and \(\proj\) as
follows. First, a component of the form
\(\unit_{\obj \times \bas} : \obj \times \bas \to \poly_{\pr}(\obj
\times \tot)\) for \(\obj \in \cart\) is determined by the equations
\begin{align*}
  \poly_{\pr}\prodpr_{1} \circ \unit_{\obj \times \bas}
  &= \wk_{\obj} \\
  \poly_{\pr}\prodpr_{2} \circ \unit_{\obj \times \bas}
  &= \proj \circ \prodpr_{2}
\end{align*}
because \(\poly_{\pr}\) sends binary products in \(\cart\) to
pullbacks over \(\bas\). For a general component
\(\unit_{\objI} : \objI \to \poly_{\pr}(\objI \times_{\bas} \tot)\)
for \((\arr : \objI \to \bas) \in \cart/\bas\), observe that
\cref{lem:object-in-slice} implies that the component
\(\unit_{\objI}\) is determined by the components
\(\unit_{\objI \times \bas}\), \(\unit_{\bas \times \bas}\) and
\(\unit_{\bas} = \proj\), since \(\poly_{\pr}\) preserves
pullbacks. From the description of \(\unit_{\obj \times \bas}\), the
component
\(\unit_{\objI} : \objI \to \poly_{\pr}(\objI \times_{\bas} \tot)\) is
determined also by the equations
\begin{align}
  \label{eq:17}
  \poly_{\pr}\prodpr_{1} \circ \unit_{\objI}
  &= \wk_{\objI} \circ (\id_{\objI}, \arr) \\
  \label{eq:18}
  \poly_{\pr}\prodpr_{2} \circ \unit_{\objI}
  &= \proj \circ \arr.
\end{align}

This motivates us to characterize the polynomial functor in terms of
\(\leg\), \(\wk\), \(\proj\) and \(\subst\). The natural
transformation \(\wk\) and the arrow \(\proj\) should satisfy some
equations. One axiomatization is as follows.

\begin{definition}
  \def\Pref{P\arabic*}

  \label{def:family-structure}
  Let \(\pr : \tot \to \bas\) be an arrow in a cartesian category
  \(\cart\). A \emph{polynomial functor for \(\pr\)} is an endofunctor
  \(\poly : \cart \to \cart\) preserving pullbacks equipped with the
  following structure:
  \begin{itemize}
  \item a natural transformation \(\leg_{\obj} : \poly\obj \to \bas\);
  \item a natural transformation
    \(\wk_{\obj} : \obj \times \bas \to \poly\obj\) over \(\bas\);
  \item an arrow \(\proj : \bas \to \poly\tot\) over \(\bas\);
  \item a natural transformation
    \(\subst_{\obj} : \poly\obj \times_{\bas} \tot \to \obj\)
  \end{itemize}
  satisfying the following axioms:
  \begin{enumerate}[label=(\Pref), ref=\Pref]
  \item \label[axiom]{axiom:1}
    \(\poly\pr \circ \proj = \wk_{\bas} \circ \diagonal_{\bas}\)
    \[
      \begin{tikzcd}
        \bas
        \arrow[r,"\proj"]
        \arrow[d,"\diagonal_{\bas}"'] &
        \poly\tot
        \arrow[d,"\poly\pr"] \\
        \bas \times \bas
        \arrow[r,"\wk_{\bas}"'] &
        \poly\bas;
      \end{tikzcd}
    \]
  \item \label[axiom]{axiom:2}
    \(\subst_{\tot} \circ (\proj \times_{\bas} \tot)\) is the
    isomorphism \(\bas \times_{\bas} \tot \cong \tot\)
    \[
      \begin{tikzcd}
        \bas \times_{\bas} \tot
        \arrow[r,"\proj \times_{\bas} \tot"]
        \arrow[dr,"\cong"'] &
        \poly\tot \times_{\bas} \tot
        \arrow[d,"\subst_{\tot}"] \\
        & \tot;
      \end{tikzcd}
    \]
  \item \label[axiom]{axiom:3}
    \(\subst_{\obj} \circ (\wk_{\obj} \times_{\bas} \tot)\) is the
    projection
    \((\obj \times \bas) \times_{\bas} \tot \cong \obj \times \tot \to
    \obj\)
    \[
      \begin{tikzcd}
        (\obj \times \bas) \times_{\bas} \tot
        \arrow[r,"\wk_{\obj} \times_{\bas} \tot"]
        \arrow[dr] &
        [4ex]
        \poly\obj \times_{\bas} \tot
        \arrow[d,"\subst_{\obj}"] \\
        & \obj;
      \end{tikzcd}
    \]
  \item \label[axiom]{axiom:4}
    \(\poly\subst_{\obj} \circ (\wk_{\poly\obj}(\id_{\poly\obj},
    \leg_{\obj}), \proj\leg_{\obj}) = \id_{\poly\obj}\), where
    \((\wk_{\poly\obj}(\id_{\poly\obj}, \leg_{\obj}),
    \proj\leg_{\obj})\) is defined using the isomorphism
    \(\poly\poly\obj \times_{\poly\bas} \poly\tot \cong
    \poly(\poly\obj \times_{\bas} \tot)\) as \(\poly\) preserves
    pullbacks
    \[
      \begin{tikzcd}
        \poly\obj
        \arrow[r,"{(\wk_{\poly\obj}(\id_{\poly\obj}, \leg_{\obj}),
          \proj\leg_{\obj})}"]
        \arrow[dr,equal] &
        [12ex]
        \poly(\poly\obj \times_{\bas} \tot)
        \arrow[d,"\poly\subst_{\obj}"] \\
        & \poly\obj.
      \end{tikzcd}
    \]
  \end{enumerate}
\end{definition}

\begin{remark}
  The intuition behind \cref{def:family-structure} will be explained
  in \cref{sec:expon-arrows-assoc}. Roughly, the data \(\wk\),
  \(\proj\) and \(\subst\) correspond to the structural rules of
  weakening, projection and substitution, respectively, of dependent
  type theory, and the axioms express the interaction of these rules.
\end{remark}

\begin{lemma}
  \label{lem:polynomial-1}
  For any exponentiable arrow \(\pr : \tot \to \bas\) in a cartesian
  category \(\cart\), the associated polynomial functor
  \(\poly_{\pr} : \cart \to \cart\) is a polynomial functor for
  \(\pr\) in the sense of \cref{def:family-structure} with \(\leg\),
  \(\wk\), \(\proj\) and \(\subst\) constructed as above.
\end{lemma}
\begin{proof}
  \Cref{axiom:1} follows from the naturality of
  \(\unit\). \Cref{axiom:2,axiom:3} follow from the special cases of
  the triangle identity \labelcref{eq:16} when \(\objI = \bas\) and
  when \(\objI = \obj \times \bas\), respectively. By
  \cref{eq:17,eq:18}, \cref{axiom:4} is the same as the triangle
  identity \labelcref{eq:15}.
\end{proof}

\begin{lemma}
  \label{lem:polynomial-3}
  Let \(\cart\) be a cartesian category, \(\pr : \tot \to \bas\) an
  arrow in \(\cart\), \(\poly : \cart \to \cart\) an endofunctor
  preserving pullbacks and \(\leg_{\obj} : \poly\obj \To \bas\) and
  \(\subst_{\obj} : \poly\obj \times_{\bas} \tot \to \obj\) natural
  transformations for \(\obj \in \cart\). Then we have a bijective
  correspondence between the following sets:
  \begin{itemize}
  \item the set of natural transformations
    \(\unit_{\objI} : \objI \to \poly(\objI \times_{\bas} \tot)\) for
    \(\objI \in \cart/\bas\) making
    \((\poly, \leg) : \cart \to \cart/\bas\) a right adjoint of
    \((\argu \times_{\bas} \tot)\) with unit \(\unit\) and counit
    \(\subst\);
  \item the set of pairs \((\wk, \proj)\) consisting of a natural
    transformation \(\wk_{\obj} : \obj \times \bas \to \poly\obj\)
    over \(\bas\) for \(\obj \in \cart\) and an arrow
    \(\proj : \bas \to \poly\tot\) over \(\bas\) making
    \((\poly, \leg, \wk, \proj, \subst)\) a polynomial functor for
    \(\pr\).
  \end{itemize}
  Concretely, for a natural transformation \(\unit_{\objI} : \objI \to
  \poly(\objI \times_{\bas} \tot)\) for \(\objI \in \cart/\bas\), the
  corresponding pair \((\wk, \proj)\) is defined by
  \begin{align}
    \label{eq:25}
    \wk_{\obj}
    &= \left(
      \begin{tikzcd}[ampersand replacement=\&]
        \obj \times \bas
        \arrow[r, "\unit_{\obj \times \bas}"] \&
        \poly((\obj \times \bas) \times_{\bas} \tot)
        \cong \poly(\obj \times \tot)
        \arrow[r, "\poly\prodpr_{1}"] \&
        \poly\obj
      \end{tikzcd}
    \right)\\
    \label{eq:26}
    \proj
    &= \left(
      \begin{tikzcd}[ampersand replacement=\&]
        \bas
        \arrow[r, "\unit_{\bas}"] \&
        \poly(\bas \times_{\bas} \tot)
        \cong \poly\tot
      \end{tikzcd}
      \right).
  \end{align}
\end{lemma}
\begin{proof}
  We have already seen in \cref{lem:polynomial-1} that
  \((\wk, \proj)\) defined as \cref{eq:25,eq:26} satisfies the axioms
  of a polynomial functor for \(\pr\).

  To give an inverse construction, let \((\wk, \proj)\) be a pair
  making \((\poly, \leg, \wk, \proj, \subst)\) a polynomial functor
  for \(\pr\). Since \(\poly\) preserves pullbacks, we can define a
  natural transformation
  \(\unit_{\objI} : \objI \to \poly(\objI \times_{\bas} \tot)\) for
  \((\arr : \objI \to \bas) \in \cart/\bas\) by the equations
  \begin{align}
    \label{eq:21}
    \poly\prodpr_{1} \circ \unit_{\objI}
    &= \wk_{\objI} \circ (\id_{\objI}, \arr) \\
    \label{eq:22}
    \poly\prodpr_{2} \circ \unit_{\objI}
    &= \proj \circ \arr.
  \end{align}
  We check that the triangle identities
  \begin{align}
    \label{eq:19}
    \poly\subst_{\obj} \circ \unit_{\poly\obj}
    &= \id_{\poly\obj}
    & (\obj \in \cart) \\
    \label{eq:20}
    \subst_{\objI \times_{\bas} \tot} \circ (\unit_{\objI} \times_{\bas}
    \tot)
    &= \id_{\objI \times_{\bas} \tot}
    & (\objI \in \cart/\bas)
  \end{align}
  are satisfied. \Cref{eq:19} is the same as \cref{axiom:4}. For
  \cref{eq:20}, it suffices to show that
  \begin{align}
    \label{eq:23}
    \prodpr_{1} \circ \subst_{\objI \times_{\bas} \tot} \circ
    (\unit_{\objI} \times_{\bas} \tot)
    &= \prodpr_{1} \\
    \label{eq:24}
    \prodpr_{2} \circ \subst_{\objI \times_{\bas} \tot} \circ
    (\unit_{\objI} \times_{\bas} \tot)
    &= \prodpr_{2}.
  \end{align}
  By the naturality of \(\subst\), the following squares commute.
  \[
    \begin{tikzcd}
      \poly\objI \times_{\bas} \tot
      \arrow[d, "\subst_{\objI}"'] &
      [6ex]
      \poly(\objI \times_{\bas} \tot) \times_{\bas} \tot
      \arrow[l, "\poly\prodpr_{1} \times_{\bas} \tot"']
      \arrow[d, "\subst_{\objI \times_{\bas} \tot}"]
      \arrow[r, "\poly\prodpr_{2} \times_{\bas} \tot"] &
      [6ex]
      \poly\tot \times_{\bas} \tot
      \arrow[d, "\subst_{\tot}"] \\
      \objI &
      \objI \times_{\bas} \tot
      \arrow[l, "\prodpr_{1}"]
      \arrow[r, "\prodpr_{2}"'] &
      \tot
    \end{tikzcd}
  \]
  Then \cref{eq:21,eq:22,axiom:2,axiom:3} imply \cref{eq:23,eq:24}.

  We show that the constructions \(\unit \mapsto (\wk, \proj)\) and
  \((\wk, \proj) \mapsto \unit\) are mutually inverses. We have
  already seen one of the identities before
  \cref{def:family-structure}: \(\unit\) is recovered from
  \((\wk, \proj)\) by \cref{eq:17,eq:18}. For the other identity, let
  \((\wk, \proj)\) be a pair making
  \((\poly, \leg, \wk, \proj, \subst)\) a polynomial functor for
  \(\pr\), and define \(\unit\) by \cref{eq:21,eq:22}. We have to show
  that \cref{eq:25,eq:26} are satisfied. \Cref{eq:26} is immediate
  from \cref{eq:22}. For \cref{eq:25}, observe that the following
  diagram commutes
  \begin{equation}
    \label{eq:27}
    \begin{tikzcd}
      \obj \times \bas
      \arrow[r, "\unit_{\obj \times \bas}"]
      \arrow[d, "{(\id_{\obj \times \bas}, \prodpr_{2})}"'] &
      \poly((\obj \times \bas) \times_{\bas} \tot)
      \arrow[r, "\cong"]
      \arrow[d, "\poly\prodpr_{1}"] &
      \poly(\obj \times \tot)
      \arrow[d, "\poly\prodpr_{1}"] \\
      \obj \times \bas \times \bas
      \arrow[r, "\wk_{\obj \times \bas}"'] &
      \poly(\obj \times \bas)
      \arrow[r, "\poly\prodpr_{1}"'] &
      \poly\obj,
    \end{tikzcd}
  \end{equation}
  where the commutativity of the left square is an instance of
  \cref{eq:21}. By the naturality of \(\wk\), the diagram
  \begin{equation}
    \label{eq:28}
    \begin{tikzcd}
      \obj \times \bas
      \arrow[r, "{(\id_{\obj \times \bas}, \prodpr_{2})}"]
      \arrow[dr, equal] &
      [4ex]
      \obj \times \bas \times \bas
      \arrow[r, "\wk_{\obj \times \bas}"]
      \arrow[d, "\prodpr_{1} \times \bas"] &
      \poly(\obj \times \bas)
      \arrow[d, "\poly\prodpr_{1}"] \\
      & \obj \times \bas
      \arrow[r, "\wk_{\obj}"'] &
      \poly\obj
    \end{tikzcd}
  \end{equation}
  commutes. Then \cref{eq:25} follows from the commutativity of
  \cref{eq:27,eq:28}.
\end{proof}

From \cref{lem:polynomial-3}, constructing a right adjoint of
\((\argu \times_{\bas} \tot)\) is equivalent to constructing a
polynomial functor for \(\pr\). Combined with \cref{prop:exp-poly}, we
have the following.

\begin{theorem}
  \label{thm:family-repr}
  For an arrow \(\pr : \tot \to \bas\) in a cartesian category
  \(\cart\), the following are equivalent:
  \begin{enumerate}
  \item \label{item:7} \(\pr\) is exponentiable;
  \item \label{item:9} there exists a polynomial functor for \(\pr\).
  \end{enumerate}
  Moreover, if this is the case, then the associated polynomial
  functor \(\poly_{\pr}\) is a polynomial functor for \(\pr\) in the
  sense of \cref{def:family-structure}, and any polynomial functor for
  \(\pr\) is isomorphic to \(\poly_{\pr}\). \qed
\end{theorem}

We also note that a cartesian functor preserves pushforwards precisely
when it commutes with the associated polynomial functors in the
following sense.

\begin{proposition}
  \label{prop:preserve-poly}
  Let \(\cart\) and \(\cartI\) be cartesian categories,
  \(\pr : \tot \to \bas\) an exponentiable arrow in \(\cart\) and
  \(\functor : \cart \to \cartI\) a cartesian functor that sends
  \(\pr\) to an exponentiable arrow. Then, the canonical natural
  transformation
  \((\functor/\bas)\pr_{*} \To (\functor\pr)_{*}(\functor/\tot)\) is
  an isomorphism if and only if the natural transformation
  \(\functor\poly_{\pr} \To \poly_{\functor\pr}\functor\) defined by
  the composite
  \[
    \begin{tikzcd}
      \cart
      \arrow[r,"\functor"]
      \arrow[d,"\tot^{*}"']
      \arrow[dr,phantom,"\cong"{description}] &
      \cartI
      \arrow[d,"(\functor\tot)^{*}"] \\
      \cart/\tot
      \arrow[r,"\functor/\tot"]
      \arrow[d,"\pr_{*}"']
      \arrow[dr,phantom,bend right,""{name=a0}]
      \arrow[dr,phantom,bend left,""'{name=a1}]
      \arrow[from=a0,to=a1,Rightarrow] &
      \cartI/\functor\tot
      \arrow[d,"(\functor\pr)_{*}"] \\
      \cart/\bas
      \arrow[r,"\functor/\bas"']
      \arrow[d,"\bas_{!}"'] &
      \cartI/\functor\bas
      \arrow[d,"(\functor\bas)_{!}"] \\
      \cart
      \arrow[r,"\functor"'] &
      \cartI
    \end{tikzcd}
  \]
  is an isomorphism.
\end{proposition}
\begin{proof}
  Clearly \((\functor\bas)_{!} : \cartI/\functor\bas \to \cartI\)
  reflects isomorphisms. \Cref{lem:object-in-slice} implies that the
  precomposition
  \[
    (\argu\tot^{*}) : \Cart(\cart/\tot, \cartI/\functor\bas) \to
    \Cart(\cart, \cartI/\functor\bas)
  \]
  reflects isomorphisms too. Therefore, if the natural transformation
  \(\functor\poly_{\pr} \To \poly_{\functor\pr}\functor\) is an
  isomorphism, so is the natural transformation
  \((\functor/\bas)\pr_{*} \To (\functor\pr)_{*}(\functor/\tot)\).
\end{proof}

\section{Exponentiable Arrows from Type Theories}
\label{sec:depend-type-theor}

The goal of this section is to construct an exponentiable arrow in a
category of theories written in a type theory. Among other aspects of
type theories, we think of a type theory as a framework for defining
\emph{theories}. In this sense, what we call a type theory can be
called a logical framework
\parencite{harper1993framework,nordstrom2001martin-lof}; see
\cref{exm:sltt,exm:mltt}. For each type theory, we introduce a
category of finite theories in \cref{sec:categories-theories}, and
construct an exponentiable arrow in the category of finite theories in
\cref{sec:expon-arrows-assoc}.

Informally, a \emph{type theory} is a formal system for deriving
judgments, but we need a formal definition of a type theory to make
the construction of the exponentiable arrow precise. A definition of a
general type theory can vary according to purposes (see
\cite{bauer2020general,isaev2017algebraic,uemura2019framework}
for some approaches to general definitions of type theories), and our
requirements for the definition of a type theory are as follows:
\begin{itemize}
\item for each type theory \(\tth\), one can form a category of finite
  theories within \(\tth\);
\item a type theory is a sufficiently rich structure for building an
  exponentiable arrow in the category of finite theories;
\item one can verify that a wide variety of concrete type theories,
  which are usually presented syntactically, are instances of the
  definition.
\end{itemize}
We give a definition of a type theory fulfilling these requirements in
\cref{sec:type-theories}. We note that, for our purpose, we need not
exclude bad type theories as long as they have enough structure, and
there are indeed ill-behaved examples of our definition which people
would not consider type theories; see \cref{rem:ill-behaved-tt}.

We assume that the reader is familiar with the syntax of dependent
type theory
\parencite{barendregt1992lambda,cartmell1978generalised,hofmann1997syntax,nordstrom1990programming}.

\subsection{Type Theories}
\label{sec:type-theories}

A type theory consists of a grammar for raw expressions and a set of
inference rules. We regard a grammar for raw expressions as a map
sending a set of symbols to a set of expressions generated by the
symbols (\cref{def:expressions}) and a set of inference rules as a map
sending a set of basic judgments to a set of derivable judgments
(\cref{def:type-theory}).

The \emph{substitution} operator is the most important structure on
sets of raw expressions for our purpose. Following
\textcite{altenkirch2010monads}, we formulate the substitution
operator using the notion of a \emph{relative monad}, which is almost
a monad but the domain is restricted.

\begin{definition}
  Let \(\cat\) be a category and \(\cat_{0} \subset \cat\) is a full
  subcategory. A \emph{relative monad on the inclusion \(\cat_{0} \to
    \cat\)} consists of the following data:
  \begin{itemize}
  \item a map on objects \(\functor : \cat_{0} \to \cat\);
  \item for any \(\obj \in \cat_{0}\), a map
    \(\unit_{\obj} : \obj \to \functor\obj\) called the \emph{unit};
  \item for any \(\obj, \objI \in \cat_{0}\), a map
    \((\argu)^{*} : \cat(\obj, \functor\objI) \to \cat(\functor\obj,
    \functor\objI)\) called the \emph{Kleisli extension}
  \end{itemize}
  satisfying the axioms analogous to those of a monad.
\end{definition}

Let us fix an infinite set \(\Var\) of variables
\(\var, \varI, \dots\). We write \(\seq{\var}\) for a finite sequence
of distinct variables \((\var_{1}, \dots, \var_{\nat})\). We will
silently coerce a sequence \(\seq{\var}\) to the set
\(\{\var_{1}, \dots, \var_{\nat}\}\). Let \(\Var^{\astd}\) denote the
set of finite sequences of distinct variables. We further assume that,
for every \(\seq{\var} \in \Var^{\astd}\), a fresh variable
\(\var_{0} \not\in \seq{\var}\) is chosen. We define a functor
\(\var_{0}^{*} : \Set^{\Var^{\astd}} \to \Set^{\Var^{\astd}}\) by
\(\var_{0}^{*}\set(\seq{\var}) = \set(\var_{0}, \seq{\var})\). The
functor \(\var_{0}^{*}\) has a left adjoint \((\var_{0})_{!}\) defined
by \((\var_{0})_{!}\set(\var_{0}, \seq{\var}) = \set(\seq{\var})\) and
\((\var_{0})_{!}\set(\seq{\varI}) = \emptyset\) for other
\(\seq{\varI}\).

\begin{definition}
  \label{def:expressions}
  A \emph{theory of expressions} consists of the following data:
  \begin{itemize}
  \item for each pair \((\set, \seq{\var})\) consisting of a family of
    sets \(\set : \Var^{\astd} \to \Set\) and a finite sequence of
    variables \(\seq{\var}\), a set \(\thexp(\set, \seq{\var})\);
  \item a structure making the currying of \(\thexp\) a monad on
    \(\Set^{\Var^{\astd}}\). We refer to the Kleisli extension of a map
    \(\map : \set_{1} \to \thexp(\set_{2})\) as
    \(\map^{\Kl} : \thexp(\set_{1}) \to \thexp(\set_{2})\);
  \item for each \(\set : \Var^{\astd} \to \Set\), a structure making
    the map \(\thexp(\set, \argu)\) a relative monad on the inclusion
    \(\Var^{\astd} \to \Set\), regarding \(\Var^{\astd}\) as a full
    subcategory of \(\Set\). We refer to the unit
    \(\seq{\var} \to \thexp(\set, \seq{\var})\) as
    \(\unit_{\set, \seq{\var}}\) and the Kleisli extension of a map
    \(\mor : \seq{\varI} \to \thexp(\set, \seq{\var})\) as
    \(\mor^{\KlI} : \thexp(\set, \seq{\varI}) \to \thexp(\set,
    \seq{\var})\);
  \item a natural transformation
    \(\lax : \thexp\var_{0}^{*} \To \var_{0}^{*}\thexp :
    \Set^{\Var^{\astd}} \to \Set^{\Var^{\astd}}\) compatible with the
    monad structure on \(\thexp\) (precisely, a natural transformation
    making
    \(\var_{0}^{*} : \Set^{\Var^{\astd}} \to \Set^{\Var^{\astd}}\) a
    lax endomorphism on the monad \(\thexp\)).
  \end{itemize}
  Moreover, these (relative) monads are required to satisfy the
  following:
  \begin{enumerate}
  \item \label[axiom]{axiom:5}
    for any map \(\map : \set_{1} \to \thexp(\set_{2})\) in
    \(\Set^{\Var^{\astd}}\) and \(\seq{\var} \in \Var^{\astd}\), we have
    \(\map^{\Kl}_{\seq{\var}} \circ \unit_{\set_{1}, \seq{\var}} =
    \unit_{\set_{2}, \seq{\var}}\)
    \[
      \begin{tikzcd}
        \seq{\var}
        \arrow[rd, "\unit_{\set_{2}, \seq{\var}}"]
        \arrow[d, "\unit_{\set_{1}, \seq{\var}}"'] \\
        \thexp(\set_{1}, \seq{\var})
        \arrow[r, "\map^{\Kl}_{\seq{\var}}"'] &
        \thexp(\set_{2}, \seq{\var});
      \end{tikzcd}
    \]
  \item \label[axiom]{axiom:6}
    for any map \(\map : \set_{1} \to \thexp(\set_{2})\) in
    \(\Set^{\Var^{\astd}}\) and any map
    \(\mor : \seq{\varI} \to \thexp(\set_{1}, \seq{\var})\), we have
    \(\map^{\Kl}_{\seq{\var}} \circ \mor^{\KlI} =
    (\map^{\Kl}_{\seq{\var}} \circ \mor)^{\KlI} \circ
    \map^{\Kl}_{\seq{\varI}}\)
    \[
      \begin{tikzcd}
        \thexp(\set_{1}, \seq{\varI})
        \arrow[r, "\mor^{\KlI}"]
        \arrow[d, "\map^{\Kl}_{\seq{\varI}}"'] &
        \thexp(\set_{1}, \seq{\var})
        \arrow[d, "\map^{\Kl}_{\seq{\var}}"] \\
        \thexp(\set_{2}, \seq{\varI})
        \arrow[r, "(\map^{\Kl}_{\seq{\var}}\mor)^{\KlI}"'] &
        \thexp(\set_{2}, \seq{\var});
      \end{tikzcd}
    \]
  \item \label[axiom]{axiom:7} the units of the monad \(\thexp\) and
    the relative monad \(\thexp(\set, \argu)\) are monomorphisms;
  \item \label[axiom]{axiom:8} the Kleisli extension \(\map \mapsto
    \map^{\Kl}\) preserves monomorphisms.
  \end{enumerate}
\end{definition}

\begin{example}
  \label{exm:type-theory}
  A lot of theories of expressions are specified grammatically. For
  example, the grammar
  \[
    \expr, \exprI ::= \var \mid \cst(\expr_{1}, \dots, \expr_{\nat})
  \]
  defines the following theory of expressions \(\thexp_{0}\). For a
  family of sets \(\set : \Var^{\astd} \to \Set\), the family of sets
  \(\thexp_{0}(\set, \argu) : \Var^{\astd} \to \Set\) is inductively
  defined as follows:
  \begin{itemize}
  \item if \(\var_{\idx} \in \seq{\var}\), then \(\var_{\idx} \in
    \thexp_{0}(\set, \seq{\var})\);
  \item if \(\cst \in \thexp_{0}(\set, \seq{\varI})\) and
    \(\mor : \seq{\varI} \to \thexp_{0}(\set, \seq{\var})\) is a map,
    then \(\cst(\mor) \in \thexp_{0}(\set, \seq{\var})\).
  \end{itemize}
  The monad structure on \(\thexp_{0}\) is given by extending a map
  \(\map : \set_{1} \to \thexp_{0}(\set_{2})\) in
  \(\Set^{\Var^{\astd}}\) to a map
  \(\map^{\Kl} : \thexp_{0}(\set_{1}) \to \thexp_{0}(\set_{2})\) by
  induction. The relative monad structure on \(\thexp(\set, \argu)\)
  is given by substitution. \(\thexp_{0}(\var_{0}^{*}\set, \argu)\) is
  the family of sets of expressions built out of the symbols
  \(\cst(\var_{0}, \seq{\var})\) of \(\set\) regarded as symbols with
  arity \(\seq{\var}\). We thus have a natural map
  \(\thexp_{0}(\var_{0}^{*}\set, \seq{\var}) \to \thexp_{0}(\set,
  (\var_{0}, \seq{\var}))\) which defines a lax morphism structure on
  \(\var_{0}^{*}\).
\end{example}

As \cref{exm:type-theory} illustrates, we think of
\(\thexp(\set, \seq{\var})\) as a set of expressions over the set of
symbols \(\set\) with variables \(\seq{\var}\). For a map
\(\mor : \seq{\varI} \to \thexp(\set, \seq{\var})\), the Kleisli
extension \(\mor^{\KlI}\) is thought of as the \emph{substitution of
  \(\mor\)}, and thus we write \({\expr}[\mor]\) for
\(\mor^{\KlI}(\expr)\). For a map
\(\map : \set_{1} \to \thexp(\set_{2})\), the Kleisli extension
\(\map^{\Kl}\) is thought of as the extension of the assignment
\(\map\) to symbols to an assignment to all
expressions. \Cref{axiom:5,axiom:6} mean that \(\map^{\Kl}\) is
identity on variables and commutes with
substitution. \Cref{axiom:7,axiom:8} are not essential, but we assume
them for notational conventions. By \cref{axiom:7}, we regard
\(\seq{\var}\) and \(\set(\seq{\var})\) as subsets of
\(\thexp(\set, \seq{\var})\). \Cref{axiom:8} implies that if
\(\set_{1} \subset \set_{2}\) then
\(\thexp(\set_{1}) \subset \thexp(\set_{2})\). For a subsequence
\(\seq{\varI} \subset \seq{\var}\), the substitution of the
restriction
\(\unit|_{\seq{\varI}} : \seq{\varI} \subset \seq{\var} \to
\thexp(\set, \seq{\var})\) is called the \emph{weakening}, and we will
omit \([\unit|_{\seq{\varI}}]\) so that an expression
\(\expr \in \thexp(\set, \seq{\varI})\) may be regarded as an
expression in \(\thexp(\set, \seq{\var})\).

\begin{definition}
  Let \(\thexp\) be a theory of expressions and
  \(\set : \Var^{\astd} \to \Set\) a family of sets. A \emph{context over
    \(\set\)} is a finite sequence of the form
  \[
    \var_{1} : \ty_{1}, \dots, \var_{\nat} : \ty_{\nat}
  \]
  where \(\var_{1}, \dots, \var_{\nat}\) are distinct variables, and
  \(\ty_{\idx} \in \thexp(\set, (\var_{1}, \dots, \var_{\idx -
    1}))\). We write \(\seq{\var} : \seq{\ty}\) for such a context. A
  \emph{statement over \((\set, \seq{\var})\)} is one of the following
  forms
  \begin{mathpar}
    \s{Ctx}
    \and
    \ty_{1} : \s{Type}
    \and
    \el_{1} : \ty_{1}
    \and
    \ty_{1} = \ty_{2} : \s{Type}
    \and
    \el_{1} = \el_{2} : \ty_{1}
  \end{mathpar}
  where \(\ty_{\idx}, \el_{\idx} \in \thexp(\set, \seq{\var})\). A
  \emph{judgment over \(\set\)} is a pair
  \(((\seq{\var} : \seq{\ty}), \judg)\) of a context
  \(\seq{\var} : \seq{\ty}\) over \(\set\) and a statement \(\judg\)
  over \((\set, \seq{\var})\).
\end{definition}

\begin{definition}
  Let \(\thexp\) be a theory of expressions. A
  \emph{\(\thexp\)-pretheory} \(\gat\) consists of the following data:
  \begin{itemize}
  \item a family of sets \(\set_{\gat} : \Var^{\astd} \to \Set\)
    equipped with a well-founded relation on
    \(\sum_{\seq{\var} \in \Var^{\astd}}\set_{\gat}(\seq{\var})\). For
    an element \(\cst \in \set_{\gat}(\seq{\var})\), we write
    \(\downset{\cst}\) for the subfamily of \(\set_{\gat}\) spanned by
    those elements below \(\cst\);
  \item for each \(\cst \in \set_{\gat}(\seq{\var})\), a pair
    \(\boundary_{\gat}(\cst) = (\ctx_{\gat}(\cst),
    \kind_{\gat}(\cst))\) where \(\ctx_{\gat}(\cst)\) is a context
    over \(\downset{\cst}\) of the form \((\seq{\var} : \seq{\ty})\)
    and \(\kind_{\gat}(\cst)\) is of either form of \(\s{Type}\),
    \(\ty_{1}\), \((\ty_{1} = \ty_{2} : \s{Type})\) or
    \((\el_{1} = \el_{2} : \ty_{1})\) with
    \(\ty_{\idx}, \el_{\idx} \in \thexp(\downset{\cst}, \seq{\var})\).
  \end{itemize}
\end{definition}
We write
\[
  \cst : \ctx \To \kind
\]
when \(\cst \in \set_{\gat}(\seq{\var})\) and
\(\boundary_{\gat}(\cst) = (\ctx, \kind)\). In practical example, the
well-founded relation on
\(\sum_{\seq{\var} \in \Var^{\astd}}\set_{\gat}(\seq{\var})\) is a
well-ordering, and in that case the \(\thexp\)-pretheory is presented
by a list like
\begin{align*}
  \cst_{1}
  &: \ctx_{1} \To \kind_{1} \\
  \cst_{2}
  &: \ctx_{2} \To \kind_{2} \\
  \cst_{3}
  &: \ctx_{3} \To \kind_{3} \\
  \vdots
\end{align*}
An element
\((\cst : \ctx \To \kind) \in \gat\) is called a \emph{symbol} when
\(\kind = \s{Type}\) or \(\kind = \ty\), and is called an
\emph{axiom} when \(\kind = (\ty_{1} = \ty_{2} : \s{Type})\) or
\(\kind = (\el_{1} = \el_{2} : \ty)\). The name of an axiom is often
irrelevant, so we write
\[
  \_ : \ctx \To \kind
\]
to mean that \(\gat\) has an axiom of the form
\(\cst : \ctx \To \kind\). For an element
\((\cst : \ctx \To \kind) \in \gat\), we define the \emph{basic
  judgment} \((\ctx(\cst), \judg(\cst))\) as
\begin{itemize}
\item \(\judg(\cst) = (\cst : \kind)\) when \(\cst\) is a symbol;
\item \(\judg(\cst) = \kind\) when \(\cst\) is an axiom.
\end{itemize}
We write \(\gat_{1} \subset \gat_{2}\) when \(\set_{\gat_{1}}\) is a
downward closed subfamily of \(\set_{\gat_{2}}\) and
\(\boundary_{\gat_{1}} = \boundary_{\gat_{2}}|_{\set_{\gat_{1}}}\).

\begin{construction}
  Let \(\thexp\) be a theory of expressions and \(\gat\) a
  \(\thexp\)-pretheory. Recall that we have an adjunction
  \((\var_{0})_{!} \adj \var_{0}^{*} : \Set^{\Var^{\astd}} \to
  \Set^{\Var^{\astd}}\) and \(\var_{0}^{*}\) is a lax endomorphism on
  the monad \(\thexp\) with natural transformation
  \(\lax : \thexp\var_{0}^{*} \To \var_{0}^{*}\thexp\). By an adjoint
  argument, the left adjoint \((\var_{0})_{!}\) is an oplax
  endomorphism on the monad \(\thexp\) with a natural transformation
  \(\oplax : (\var_{0})_{!}\thexp \To \thexp(\var_{0})_{!}\), and thus
  \((\var_{0})_{!}\) extends to an endofunctor on the Kleisli category
  for \(\thexp\): for a map \(\map : \set_{1} \to \thexp(\set_{2})\),
  we have a map
  \((\var_{0})_{!}\map : (\var_{0})_{!}\set_{1} \to
  \thexp((\var_{0})_{!}\set_{2})\) by the composite
  \[
    (\var_{0})_{!}\set_{1} \xrightarrow{(\var_{0})_{!}\map}
    (\var_{0})_{!}\thexp(\set_{2}) \xrightarrow{\oplax}
    \thexp((\var_{0})_{!}\set_{2}).
  \]
  Furthermore, \(\oplax\) becomes natural with respect to morphisms in
  the Kleisli category: for any map \(\map : \set_{1} \to
  \thexp(\set_{2})\), the diagram
  \[
    \begin{tikzcd}
      (\var_{0})_{!}\thexp(\set_{1})
      \arrow[r, "\oplax"]
      \arrow[d, "(\var_{0})_{!}\map^{\Kl}"'] &
      \thexp((\var_{0})_{!}\set_{1})
      \arrow[d, "((\var_{0})_{!}\map)^{\Kl}"] \\
      (\var_{0})_{!}\thexp(\set_{2})
      \arrow[r, "\oplax"'] &
      \thexp((\var_{0})_{!}\set_{2})
    \end{tikzcd}
  \]
  commutes. By transpose, we have a natural transformation
  \[
    \hypsz{\var_{0}} : \thexp(\set) \to
    \var_{0}^{*}\thexp((\var_{0})_{!}\set)
  \]
  such that the diagram
  \begin{equation}
    \label[diagram]{eq:29}
    \begin{tikzcd}
      \thexp(\set_{1})
      \arrow[r, "\hypsz{\var_{0}}"]
      \arrow[d, "\map^{\Kl}"'] &
      \var_{0}^{*}\thexp((\var_{0})_{!}\set_{1})
      \arrow[d, "\var_{0}^{*}((\var_{0})_{!}\map)^{\Kl}"] \\
      \thexp(\set_{2})
      \arrow[r, "\hypsz{\var_{0}}"'] &
      \var_{0}^{*}\thexp((\var_{0})_{!}\set_{2})
    \end{tikzcd}
  \end{equation}
  commutes. For a fresh symbol \(\ty_{0}\), we define
  \(\gat^{\ty_{0}}\) to be the following \(\thexp\)-pretheory.
  \begin{align*}
    \ty_{0}
    &: () \To \s{Type} \\
    \hypsz{\var_{0}}\cst
    &: (\var_{0} : \ty_{0}, \hypsz{\var_{0}}\ctx) \To
      \hypsz{\var_{0}}\kind
    & ((\cst : \ctx \To \kind) \in \gat)
  \end{align*}
\end{construction}

\begin{definition}
  \label{def:type-theory}

  \def\TTref{TT\arabic*}

  A \emph{type theory} \(\tth\) consists of the following data:
  \begin{itemize}
  \item a theory of expressions \(\thexp_{\tth}\);
  \item for each \(\thexp_{\tth}\)-pretheory \(\gat\), a set
    \(\derivable_{\tth}(\gat)\) of judgments over \(\set\). Judgments
    in \(\derivable_{\tth}(\gat)\) are called \emph{derivable
      judgments}. We write \(\ctx \vdash_{\gat} \judg\) when
    \((\ctx; \judg) \in \derivable_{\tth}(\gat)\).
  \end{itemize}
  Furthermore, it is required to satisfy the following conditions:
  \begin{enumerate}[label=(\TTref),ref=\TTref]
  \item \label[axiom]{axiom:9} \(\derivable_{\tth}(\gat)\) is closed
    under the rules listed in \cref{fig:derivable-judgments} and rules
    for \(\ty_{1} = \ty_{2} : \s{Type}\) and
    \(\el_{1} = \el_{2} : \ty\) to be congruence relations;
  \item \label[axiom]{axiom:10} derivable judgments are stable under
    \emph{hypothesizing}: for any \(\thexp_{\tth}\)-pretheory \(\gat\)
    and for any symbol \(\ty_{0} \not\in \gat\), if
    \(\ctx \vdash_{\gat} \judg\) then
    \(\var_{0} : \ty_{0}, \hypsz{\var_{0}}\ctx \vdash_{\gat^{\ty_{0}}}
    \hypsz{\var_{0}}\judg\);
  \item \label[axiom]{axiom:11} for any \(\thexp_{\tth}\)-pretheories
    \(\gat_{1}\) and \(\gat_{2}\), and for any map
    \(\itpr : \set_{\gat_{1}} \to \thexp_{\tth}(\set_{\gat_{2}})\), if
    \(\itpr^{\Kl}(\ctx) \vdash_{\gat_{2}} \itpr^{\Kl}(\judg(\cst))\)
    for all \((\cst : \ctx \To \kind) \in \gat_{1}\), then
    \(\itpr^{\Kl}(\ctx) \vdash_{\gat_{2}} \itpr^{\Kl}(\judg)\) for all
    \(\ctx \vdash_{\gat_{1}} \judg\);
  \item \label[axiom]{axiom:12} for any \(\thexp_{\tth}\)-pretheories
    \(\gat_{1}\) and \(\gat_{2}\), and for any maps
    \(\itpr_{1}, \itpr_{2} : \set_{\gat_{1}} \to
    \thexp_{\tth}(\set_{\gat_{2}})\), if
    \(\itpr_{1}^{\Kl}(\ctx) \vdash_{\gat_{2}} \itpr_{1}(\cst) =
    \itpr_{2}(\cst) : \itpr_{1}^{\Kl}(\kind)\) for all symbols
    \((\cst : \ctx \To \kind) \in \gat_{1}\), then
    \(\itpr_{1}^{\Kl}(\ctx) \vdash_{\gat_{2}} \itpr_{1}^{\Kl}(\ty) =
    \itpr_{2}^{\Kl}(\ty) : \s{Type}\) for all
    \(\ctx \vdash_{\gat_{1}} \ty : \s{Type}\), and
    \(\itpr_{1}^{\Kl}(\ctx) \vdash_{\gat_{2}} \itpr_{1}^{\Kl}(\el) =
    \itpr_{2}^{\Kl}(\el) : \itpr_{1}^{\Kl}(\ty)\) for all
    \(\ctx \vdash_{\gat_{1}} \el : \ty\).
  \end{enumerate}
  \begin{figure}
    \begin{mathpar}
      \inferrule
      {
        \ctx \vdash_{\gat} \s{Ctx}
      }
      {\ctx \vdash_{\gat} \cst : \s{Type}}
      \ ((\cst : \ctx \To \s{Type}) \in \gat)
      \and
      \inferrule
      {
        \ctx \vdash_{\gat} \ty : \s{Type}
      }
      {\ctx \vdash_{\gat} \cst : \ty}
      \ ((\cst : \ctx \To \ty) \in \gat)
      \and
      \inferrule
      {
        \ctx \vdash_{\gat} \ty_{1} : \s{Type} \\
        \ctx \vdash_{\gat} \ty_{2} : \s{Type}
      }
      {\ctx \vdash_{\gat} \ty_{1} = \ty_{2} : \s{Type}}
      \ ((\_ : \ctx \To \ty_{1} = \ty_{2} : \s{Type}) \in \gat)
      \and
      \inferrule
      {
        \ctx \vdash_{\gat} \el_{1} : \ty \\
        \ctx \vdash_{\gat} \el_{2} : \ty
      }
      {\ctx \vdash_{\gat} \el_{1} = \el_{2} : \ty}
      \ ((\_ : \ctx \To \el_{1} = \el_{2} : \ty) \in \gat)
      \and
      \inferrule
      { }
      {{} \vdash_{\gat} \s{Ctx}}
      \and
      \inferrule
      {
        \ctx \vdash_{\gat} \ty : \s{Type}
      }
      {\ctx, \var : \ty \vdash_{\gat} \s{Ctx}}
      \ (\var \not\in \ctx)
      \and
      \inferrule[Weakening]
      {
        \ctx \vdash_{\gat} \ty : \s{Type} \\
        \ctx, \ctxI \vdash_{\gat} \judg
      }
      {\ctx, \var : \ty, \ctxI \vdash_{\gat} \judg}
      \ (\var \not\in \ctx, \ctxI)
      \and
      \inferrule[Projection]
      {
        \ctx \vdash_{\gat} \ty : \s{Type}
      }
      {\ctx, \var : \ty \vdash_{\gat} \var : \ty}
      \ (\var \not\in \ctx)
      \and
      \inferrule[Substitution]
      {
        \ctx \vdash_{\gat} \el : \ty \\
        \ctx, \var : \ty, \ctxI \vdash_{\gat} \judg
      }
      {\ctx, {\ctxI}[\el/\var] \vdash_{\gat} {\judg}[\el/\var]}
    \end{mathpar}
    \caption{Derivable judgments}
    \label{fig:derivable-judgments}
  \end{figure}
\end{definition}

\begin{example}
  \label{exm:tth-0}
  We define a type theory \(\tth_{0}\) as
  follows. \(\thexp_{\tth_{0}}\) is \(\thexp_{0}\) (see
  \cref{exm:type-theory}). \(\derivable_{\tth_{0}}(\gat)\) for a
  \(\thexp_{0}\)-pretheory \(\gat\) is the smallest set of judgments
  closed under the rules in
  \cref{fig:derivable-judgments}. \Cref{axiom:10,axiom:11,axiom:12}
  are verified by induction on derivation. \(\tth_{0}\) is the
  dependent type theory without any type constructors.
\end{example}

\begin{example}
  \label{exm:tth-Pi}
  One can easily define type theories with various type
  constructors. For example, we define a type theory \(\tth_{\Pi}\)
  with dependent function types (\(\Pi\)-types) by extending the
  grammar for raw expressions of \(\thexp_{0}\) as
  \[
    \expr, \exprI ::= \dots \mid \textstyle{\prod_{\var :
        \expr}\exprI} \mid \lambda(\var : \expr).\exprI \mid
    \expr\exprI,
  \]
  where the variable \(\var\) in the expression \(\var.\exprI\) is
  considered to be bound, and requiring the set of derivable judgments
  to be closed under the additional inference rules listed in
  \cref{fig:pi-types}. The type annotation in
  \(\lambda(\var : \ty).\elI\) is often omitted, and we write \(\ty
  \to \tyI\) for \(\prod_{\var : \ty}\tyI\) when \(\tyI\) does not
  contain \(\var\) as a free variable.
  \begin{figure}
    \begin{mathpar}
      \inferrule
      {\ctx \vdash \ty : \s{Type} \\
        \ctx, \var : \ty \vdash \tyI : \s{Type}}
      {\ctx \vdash \textstyle{\prod_{\var : \ty}\tyI} : \s{Type}}
      \and
      \inferrule
      {\ctx, \var : \ty \vdash \elI : \tyI}
      {\ctx \vdash \lambda(\var : \ty).\elI : \textstyle{\prod_{\var :
          \ty}\tyI}}
      \and
      \inferrule
      {\ctx \vdash \elI : \textstyle{\prod_{\var : \ty}\tyI} \\
        \ctx \vdash \el : \ty}
      {\ctx \vdash \elI\el : {\tyI}[\el/\var]}
      \and
      \inferrule
      {\ctx, \var : \ty \vdash \elI : \tyI \\
        \ctx \vdash \el : \ty}
      {\ctx \vdash (\lambda(\var : \ty).\elI)\el = {\elI}[\el/\var] : {\tyI}[\el/\var]}
      \and
      \inferrule
      {\ctx \vdash \elI : \textstyle{\prod_{\var : \ty}\tyI}}
      {\ctx \vdash \lambda(\var : \ty).\elI\var = \elI :
        \textstyle{\prod_{\var : \ty}\tyI}}
      \ (\var \not\in \ctx)
    \end{mathpar}
    \caption{Rules for \(\Pi\)-types}
    \label{fig:pi-types}
  \end{figure}
\end{example}

\begin{remark}
  When we specify a type theory by a set of inference rules as in
  \cref{exm:tth-0,exm:tth-Pi}, \cref{axiom:11,axiom:12} are verified
  by induction on derivation. The stability under hypothesizing
  (\cref{axiom:10}) is also verified by induction on derivation,
  provided that if
  \[
    \inferrule
    {
      \ctx_{1} \vdash \judg_{1} \\
      \dots \\
      \ctx_{\nat} \vdash \judg_{\nat}
    }
    {\ctx \vdash \judg}
  \]
  is an instance of an inference rule, then
  \[
    \inferrule
    {
      \var_{0} : \ty_{0}, \hypsz{\var_{0}}\ctx_{1} \vdash
      \hypsz{\var_{0}}\judg_{1} \\
      \dots \\
      \var_{0} : \ty_{0}, \hypsz{\var_{0}}\ctx_{\nat} \vdash
      \hypsz{\var_{0}}\judg_{\nat}
    }
    {\var_{0} : \ty_{0}, \hypsz{\var_{0}}\ctx \vdash \hypsz{\var_{0}}\judg}
  \]
  is also an instance of the inference rule. This condition is
  satisfied both in \cref{exm:tth-0,exm:tth-Pi}.
\end{remark}

\begin{remark}
  \label{rem:ill-behaved-tt}
  We do not exclude ill-behaved inference rules. For an extreme
  example, the system in which all judgments are derivable is a type
  theory in the sense of \cref{def:type-theory}. This is not a problem
  for our purpose of constructing an exponentiable arrow in a category
  of theories: for an ill-behaved type theory, the category of
  theories will just get degenerate.
\end{remark}

\begin{definition}
  Let \(\tth\) be a type theory. A \emph{\(\tth\)-theory} is a
  \(\thexp_{\tth}\)-pretheory \(\gat\) satisfying the following
  \emph{well-formedness conditions}:
  \begin{itemize}
  \item \(\ctx \vdash_{\downset{\cst}} \s{Ctx}\) for any
    \((\cst : \ctx \To \s{Type}) \in \gat\);
  \item \(\ctx \vdash_{\downset{\cst}} \ty : \s{Type}\) for any
    \((\cst : \ctx \To \ty) \in \gat\);
  \item \(\ctx \vdash_{\downset{\cst}} \ty_{1} : \s{Type}\) and
    \(\ctx \vdash_{\downset{\cst}} \ty_{2} : \s{Type}\) for any
    \((\cst : \ctx \To \ty_{1} = \ty_{2} : \s{Type}) \in \gat\);
  \item \(\ctx \vdash_{\downset{\cst}} \el_{1} : \ty\) and \(\ctx
    \vdash_{\downset{\cst}} \el_{2} : \ty\) for any \((\cst : \ctx \To
    \el_{1} = \el_{2} : \ty) \in \gat\).
  \end{itemize}
\end{definition}

\begin{example}
  \label{exm:t0-theory}
  A \(\tth_{0}\)-theory is essentially the same as a generalized
  algebraic theory \parencite{cartmell1978generalised}. The difference
  is that in Cartmell's definition the set of symbols and axioms is
  not equipped with a well-founded relation. However, one can define a
  canonical well-founded relation on symbols and axioms:
  \(\cst' < \cst\) if \(\cst'\) appears in the derivation tree for the
  well-formedness condition for \(\cst\). We will also see that
  \(\tth\)-theories with the same symbols and axioms but with
  different well-founded relations are identified
  (\cref{rem:on-theories}).
\end{example}

\begin{example}
  \label{exm:gat-of-cat}
  The \emph{\(\tth_{0}\)-theory of categories} consists of the
  following data.
  \begin{align*}
    \thcatOb
    &: () \To \s{Type} \\
    \thcatHom
    &: (\var_{1} : \thcatOb, \var_{2} : \thcatOb) \To \s{Type} \\
    \thcatid
    &: (\var : \thcatOb) \To \thcatHom(\var, \var) \\
    \thcatcomp
    & : (\var_{1} : \thcatOb, \var_{2} : \thcatOb, \var_{3}
      : \thcatOb, \varI_{1} : \thcatHom(\var_{1}, \var_{2}), \varI_{2} :
      \thcatHom(\var_{2}, \var_{3})) \To \thcatHom(\var_{1}, \var_{3})
    \\
    \_
    &: (\var_{1} : \thcatOb, \var_{2} : \thcatOb, \varI :
      \thcatHom(\var_{1}, \var_{2})) \To \thcatcomp(\var_{1}, \var_{1},
      \var_{2}, \thcatid(\var_{1}), \varI) = \varI \\
    \_
    &: (\var_{1} : \thcatOb, \var_{2} : \thcatOb, \varI :
      \thcatHom(\var_{1}, \var_{2})) \To \thcatcomp(\var_{1}, \var_{2},
      \var_{2}, \varI, \thcatid(\var_{2})) = \varI \\
    \_
    &: (\var_{1} : \thcatOb, \var_{2} : \thcatOb, \var_{3} :
      \thcatOb, \var_{4} : \thcatOb, \varI_{1} : \thcatHom(\var_{1},
      \var_{2}), \varI_{2} : \thcatHom(\var_{2}, \var_{3}), \varI_{3}
      : \thcatHom(\var_{3}, \var_{4})) \\
    &\quad \To \thcatcomp(\var_{1}, \var_{3},
      \var_{4}, \thcatcomp(\var_{1}, \var_{2}, \var_{3}, \varI_{1},
      \varI_{2}), \varI_{3}) = \thcatcomp(\var_{1}, \var_{2}, \var_{4},
      \varI_{1}, \thcatcomp(\var_{2}, \var_{3}, \var_{4}, \varI_{2},
      \varI_{3}))
  \end{align*}
  This is read as follows:
  \begin{itemize}
  \item \(\thcatOb\) is a type of objects;
  \item \(\thcatHom(\var_{1}, \var_{2})\) is type of morphisms from
    \(\var_{1}\) to \(\var_{2}\) when \(\var_{1}\) and \(\var_{2}\)
    are elements of \(\thcatOb\);
  \item \(\thcatid(\var)\) is an element of \(\thcatHom(\var, \var)\)
    representing the identity on \(\var\) when \(\var\) is an element
    of \(\thcatOb\). The symbol \(\thcatcomp\) represents the
    composition operator;
  \item
    \(\thcatcomp(\var_{1}, \var_{1}, \var_{2}, \thcatid(\var_{1}),
    \varI)\) and \(\varI\) are equal when \(\var_{1}\) and
    \(\var_{2}\) are elements of \(\thcatOb\) and \(\varI\) is an
    element of \(\thcatHom(\var_{1}, \var_{2})\). The other equations
    are similar.
  \end{itemize}
\end{example}

Type theories with \(\Pi\)-types are often called \emph{logical
  frameworks} \parencite{harper1993framework,nordstrom2001martin-lof}
and useful for encoding type theories. \(\Pi\)-types in logical
frameworks are used for representing variable binding in target type
theories.

\begin{example}
  \label{exm:sltt}
  \def\lam{l}
  \def\app{\alpha}
  \def\fn{\mathrm{Fun}}

  The simply typed \(\lambda\)-calculus is encoded in the following
  \(\tth_{\Pi}\)-theory.
  \begin{align*}
    \Ty
    &: () \To \s{Type} \\
    \El
    &: (\ty : \Ty) \To \s{Type} \\
    \fn
    &: (\ty : \Ty, \tyI : \Ty) \To \Ty \\
    \lam
    &: (\ty : \Ty, \tyI : \Ty, \elI : \El(\ty) \to \El(\tyI))
      \To \El(\fn(\ty, \tyI)) \\
    \app
    &: (\ty : \Ty, \tyI : \Ty, \elI : \El(\fn(\ty, \tyI)), \el :
      \El(\ty)) \To \El(\tyI) \\
    \_
    &: (\ty : \Ty, \tyI : \Ty, \elI : \El(\ty) \to \El(\tyI),
      \el : \El(\ty)) \To \app(\ty, \tyI, \lam(\ty, \tyI, \elI), \el)
      = \elI\el \\
    \_
    &: (\ty : \Ty, \tyI : \Ty, \elI : \El(\fn(\ty, \tyI))) \To
      \lam(\ty, \tyI, \lambda\var.\app(\ty, \tyI, \elI,
      \var)) = \elI
  \end{align*}
  \(\Ty\) and \(\El(\ty)\) represent the sets of types and of terms of
  type \(\ty\), respectively, in the simply typed
  \(\lambda\)-calculus. \(\fn(\ty, \tyI)\) represents the type of
  functions from \(\ty\) to \(\tyI\). Notice that the
  \(\lambda\)-abstraction is then represented by the higher-order
  function \(\lam\).
\end{example}

\begin{example}
  \label{exm:mltt}

  \def\Nat{N}
  \def\zero{0}
  \def\succ{s}
  \def\rec{r}

  Martin-L{\"o}f type theory can also be encoded in a
  \(\tth_{\Pi}\)-theory. We refer the reader to
  \parencite{nordstrom2001martin-lof} for details. Like
  \cref{exm:sltt}, we first introduce two symbols
  \begin{align*}
    \Ty
    &: () \To \s{Type} \\
    \El
    &: (\ty : \Ty) \To \s{Type}.
  \end{align*}
  Dependent function types in Martin-L{\"o}f type theory is encoded in
  the same way as function types in the simply typed
  \(\lambda\)-calculus. To encode the type of natural numbers, we add
  the following symbols.
  \begin{align*}
    \Nat
    &: () \To \Ty \\
    \zero
    &: () \To \El(\Nat) \\
    \succ
    &: (\nat : \El(\Nat)) \To \El(\Nat) \\
    \rec
    &: (\nat : \El(\Nat), \ty : \El(\Nat) \to \Ty, \el_{\zero} :
      \El(\ty\zero), \\
    &\quad \el_{\succ} : \textstyle{\prod_{\var :
      \El(\Nat)}\El(\ty\var) \to \El(\ty(\succ(\var)))}) \To
      \El(\ty\nat) \\
    \_
    &: (\ty : \El(\Nat) \to \Ty, \el_{\zero} : \El(\ty\zero), \el_{\succ} :
      \textstyle{\prod_{\var : \El(\Nat)}\El(\ty\var) \to
      \El(\ty(\succ(\var)))}) \\
    &\quad \To \rec(\zero, \ty, \el_{\zero}, \el_{\succ}) =
      \el_{\zero} \\
    \_
    &: (\nat : \El(\Nat), \ty : \El(\Nat) \to \Ty, \el_{\zero} :
      \El(\ty\zero), \el_{\succ} : \textstyle{\prod_{\var :
      \El(\Nat)}\El(\ty\var) \to \El(\ty(\succ(\var)))}) \\
    &\quad \To \rec(\succ(\nat), \ty, \el_{\zero}, \el_{\succ}) =
      \el_{\succ}\nat(\rec(\nat, \ty, \el_{\zero}, \el_{\succ}))
  \end{align*}
  One can similarly encode other inductive types.
\end{example}

\subsection{Categories of Theories}
\label{sec:categories-theories}

Let \(\tth\) denote a type theory.

\begin{definition}
  Let \(\gat_{1}\) and \(\gat_{2}\) be \(\tth\)-theories. An
  \emph{interpretation from \(\gat_{1}\) to \(\gat_{2}\)} is a map
  \(\itpr : \set_{\gat_{1}} \to \thexp_{\tth}(\set_{\gat_{2}})\) in
  \(\Set^{\Var^{\astd}}\) such that
  \(\itpr^{\Kl}(\ctx) \vdash_{\gat_{2}} \itpr^{\Kl}(\judg(\cst))\) for
  any \((\cst : \ctx \To \kind) \in \gat_{1}\). Two interpretations
  \(\itpr_{1}, \itpr_{2} : \gat_{1} \to \gat_{2}\) are said to be
  \emph{equivalent} if
  \(\itpr_{1}^{\Kl}(\ctx) \vdash_{\gat_{2}} \itpr_{1}(\cst) =
  \itpr_{2}(\cst) : \itpr_{1}^{\Kl}(\kind)\) for any symbol
  \((\cst : \ctx \To \kind)\) in \(\gat_{1}\).
\end{definition}

\begin{remark}
  By \cref{axiom:11} of \cref{def:type-theory}, for an interpretation
  \(\itpr : \gat_{1} \to \gat_{2}\), if
  \(\ctx \vdash_{\gat_{1}} \judg\) then
  \(\itpr^{\Kl}(\ctx) \vdash_{\gat_{2}} \itpr^{\Kl}(\judg)\). By
  \cref{axiom:12}, for equivalent interpretations
  \(\itpr_{1}, \itpr_{2} : \gat_{1} \to \gat_{2}\), if
  \(\ctx \vdash_{\gat_{1}} \expr : \kind\) then
  \(\itpr_{1}^{\Kl}(\ctx) \vdash_{\gat_{2}} \itpr_{1}^{\Kl}(\expr) =
  \itpr_{2}^{\Kl}(\expr) : \itpr_{1}^{\Kl}(\kind)\).
\end{remark}

\begin{remark}
  \label{rem:on-theories}
  From the definition of equivalence of interpretations, only values
  at symbols are relevant. Therefore, when defining an interpretation,
  we only specify values at symbols (values at axioms can be
  arbitrary). The well-founded relation on symbols and axioms of a
  \(\tth\)-theory is also irrelevant: if \(\gat\) and \(\gat'\) differ
  only in well-founded relations on symbols and axioms, then the
  identity map defines an isomorphism \(\gat \cong \gat'\).
\end{remark}

\begin{example}
  Let \(\gat_{1}\) be the \emph{\(\tth_{0}\)-theory of monoids}
  written as follows.
  \begin{align*}
    \thmonMon
    &: () \To \s{Type} \\
    \thmonuni
    &: () \To \thmonMon \\
    \thmonmul
    &: (\varI_{1} : \thmonMon, \varI_{2} : \thmonMon) \To \thmonMon \\
    \_
    &: (\varI : \thmonMon) \To \thmonmul(\thmonuni, \varI) = \varI \\
    \_
    &: (\varI : \thmonMon) \To \thmonmul(\varI, \thmonuni) = \varI \\
    \_
    &: (\varI_{1} : \thmonMon, \varI_{2} : \thmonMon, \varI_{3} : \thmonMon) \To
      \thmonmul(\thmonmul(\varI_{1}, \varI_{2}), \varI_{3}) = \thmonmul(\varI_{1},
      \thmonmul(\varI_{2}, \varI_{3}))
  \end{align*}
  Let \(\gat_{2}\) be the extension of the \(\tth_{0}\)-theory of
  categories (\cref{exm:gat-of-cat}) with a symbol
  \[
    \thcatbase : () \To \thcatOb.
  \]
  We have an interpretation \(\gat_{1} \to \gat_{2}\) as follows.
  \begin{align*}
    \thmonMon
    &\mapsto ({} \vdash \thcatHom(\thcatbase, \thcatbase) : \s{Type}) \\
    \thmonuni
    &\mapsto ({} \vdash \thcatid(\thcatbase) : \thcatHom(\thcatbase,
      \thcatbase)) \\
    \thmonmul
    &\mapsto (\varI_{1} : \thcatHom(\thcatbase, \thcatbase), \varI_{2}
      : \thcatHom(\thcatbase, \thcatbase) \vdash
      \thcatcomp(\thcatbase, \thcatbase, \thcatbase, \varI_{1},
      \varI_{2}) : \thcatHom(\thcatbase, \thcatbase))
  \end{align*}
  Note that this is a formal treatment of the fact that the set of
  endomorphisms at an object in a category is a monoid.
\end{example}

\begin{notation}
  We denote by \(\DTT_{\tth}\) the category whose objects are the
  \(\tth\)-theories and morphisms are the equivalence classes of the
  interpretations.
\end{notation}

\begin{definition}
  We say a \(\tth\)-theory \(\gat\) is \emph{finite} if the set
  \(\sum_{\seq{\var} \in \Var^{\astd}}\set_{\gat}(\seq{\var})\) of
  symbols and axioms is finite. We denote by \(\thdtt_{\tth}\) the
  opposite of the full subcategory of \(\DTT_{\tth}\) consisting of
  finite \(\tth\)-theories.
\end{definition}

\begin{proposition}
  \label{prop:dtt-cocomplete}
  The category \(\DTT_{\tth}\) is cocomplete: coproducts are given by
  disjoint union; coequalizers are obtained by adjoining equational
  axioms.
\end{proposition}
\begin{proof}
  This is straightforward, but we need to be careful to construct
  coequalizers. Let \(\itpr_{1}, \itpr_{2} : \gat_{1} \to \gat_{2}\)
  be interpretations between \(\tth\)-theories. We define \(\gat_{3}\)
  to be the \(\thexp_{\tth}\)-pretheory extending \(\gat_{2}\) with
  the axioms
  \[
    \axiom_{\cst} : \itpr_{1}^{\Kl}(\ctx) \To \itpr_{1}(\cst) =
    \itpr_{2}(\cst) : \itpr_{1}^{\Kl}(\kind)
  \]
  for all symbols \((\cst : \ctx \To \kind) \in \gat_{1}\), and with
  relation \(\cst_{2} < \axiom_{\cst}\) for any
  \(\cst_{2} \in \gat_{2}\) and \(\axiom_{\cst} < \axiom_{\cst'}\) for
  \(\cst < \cst'\) in \(\gat_{1}\). We then have to check the
  well-formedness condition
  \begin{equation}
    \label{eq:30}
    \itpr_{1}^{\Kl}(\ctx) \vdash_{\downset{\axiom_{\cst}}}
    \itpr_{2}(\cst) : \itpr_{1}^{\Kl}(\kind).
  \end{equation}
  This is achieved by well-founded induction on \(\cst \in
  \gat_{1}\). Suppose that the well-formedness condition
  \labelcref{eq:30} is satisfied for any \(\cst'\) below
  \(\cst\). Then we have
  \(\itpr_{1}^{\Kl}(\ctx') \vdash_{\downset{\axiom_{\cst'}}}
  \itpr_{1}(\cst') = \itpr_{2}(\cst') : \itpr_{1}^{\Kl}(\kind')\) for
  any symbol \((\cst' : \ctx' \To \kind')\) below \(\cst\). This means
  that \(\itpr_{1}\) and \(\itpr_{2}\) determine equivalent
  interpretations \(\downset{\cst} \to \downset{\axiom_{\cst}}\), and
  thus we have
  \({} \vdash_{\downset{\axiom_{\cst}}}\itpr_{1}^{\Kl}(\ctx) =
  \itpr_{2}^{\Kl}(\ctx)\) and
  \(\itpr_{1}^{\Kl}(\ctx) \vdash_{\downset{\axiom_{\cst}}}
  \itpr_{1}^{\Kl}(\kind) = \itpr_{2}^{\Kl}(\kind)\). By rewriting
  along these equalities in
  \(\itpr_{2}^{\Kl}(\ctx) \vdash_{\downset{\axiom_{\cst}}}
  \itpr_{2}(\cst) : \itpr_{2}^{\Kl}(\kind)\), we obtain \cref{eq:30}.
\end{proof}

\begin{corollary}
  Finite \(\tth\)-theories are closed under finite
  colimits. Consequently, the category \(\thdtt_{\tth}\) has finite
  limits.
\end{corollary}

Pushouts of inclusions of \(\tth\)-theories have simpler descriptions.

\begin{proposition}
  \label{prop:theories-pushouts}
  Let \(\gat_{2}\) be a \(\tth\)-theory, \(\gat_{1} \subset \gat_{2}\)
  a subtheory, and \(\itpr : \gat_{1} \to \gat_{1}'\) an
  interpretation. Let \(\gat_{2}'\) be the pushout of \(\gat_{2}\)
  along \(\itpr\).
  \[
    \begin{tikzcd}
      \gat_{1}
      \arrow[r, "\itpr"]
      \arrow[d, hook] &
      \gat_{1}'
      \arrow[d] \\
      \gat_{2}
      \arrow[r] &
      \gat_{2}'
    \end{tikzcd}
  \]
  Then \(\gat_{2}'\) is isomorphic to the \(\tth\)-theory consisting
  of the following data:
  \begin{itemize}
  \item \(\cst : \ctx \To \kind\) for each \((\cst : \ctx \To \kind)
    \in \gat_{1}'\);
  \item
    \(\widetilde{\itpr}(\cst) : \widetilde{\itpr}^{\Kl}(\ctx) \To
    \widetilde{\itpr}^{\Kl}(\kind)\) for each
    \((\cst : \ctx \To \kind) \in \gat_{2} - \gat_{1}\), where
    \(\widetilde{\itpr} : \set_{\gat_{2}} \to
    \thexp_{\tth}(\set_{\gat_{1}'} + (\set_{\gat_{2}} -
    \set_{\gat_{1}}))\) is defined by
    \[
      \widetilde{\itpr}(\cst) = \left\{
        \begin{array}{ll}
          \itpr(\cst) & (\cst \in \gat_{1}) \\
          \cst & (\cst \in \gat_{2} - \gat_{1}).
        \end{array}
      \right.
    \]
  \end{itemize}
\end{proposition}
\begin{proof}
  This is straightforward, but in order to verify that \(\gat_{2}'\)
  is indeed a \(\tth\)-theory, we use well-founded induction on
  \(\cst \in \gat_{2}\). We claim that, for any \(\cst \in \gat_{2}\),
  if \(\ctx \vdash_{\downset{\cst}} \judg\) then
  \(\widetilde{\itpr}^{\Kl}(\ctx)
  \vdash_{\downset{\widetilde{\itpr}(\cst)}}
  \widetilde{\itpr}^{\Kl}(\judg)\), and then in particular the
  well-formedness condition for \(\widetilde{\itpr}(\cst)\) for
  \(\cst \in \gat_{2} - \gat_{1}\) is satisfied. Suppose that the
  claim is satisfied for any \(\cst' \in \gat_{2}\) below
  \(\cst\). Then \(\downset{\widetilde{\itpr}(\cst)}\) is a
  \(\tth\)-theory and \(\widetilde{\itpr}\) determines an
  interpretation
  \(\downset{\cst} \to \downset{\widetilde{\itpr}(\cst)}\), and thus
  \(\ctx \vdash_{\downset{\cst}} \judg\) implies
  \(\widetilde{\itpr}^{\Kl}(\ctx)
  \vdash_{\downset{\widetilde{\itpr}(\cst)}}
  \widetilde{\itpr}^{\Kl}(\judg)\).
\end{proof}

\begin{remark}
  \label{rem:locally-pres}
  In the following, we assume that the reader is familiar with the
  theory of locally presentable categories \parencite[Chapter
  1]{adamek1994locally}. A wide variety of type theories including
  \(\tth_{0}\) are inductively defined by finitary grammar and
  inference rules with finitely many premises. For such a type theory
  \(\tth\), we may assume that, for any \(\tth\)-theory \(\gat\) and
  for any \(\cst \in \gat\), the set \(\downset{\cst}\) is finite:
  otherwise replace the well-founded relation on symbols and axioms by
  \(\cst' < \cst\) defined as \(\cst'\) appears in the derivation tree
  for the well-formedness condition for \(\cst\) (see also
  \cref{exm:t0-theory}); by assumption the derivation tree is
  finite. This implies that every \(\tth\)-theory is written as a
  directed colimit of finite \(\tth\)-theories. In other words, the
  category \(\DTT_{\tth}\) is locally finitely presentable, and the
  finitely presentable objects of \(\DTT_{\tth}\) are the finite
  \(\tth\)-theories. In that case, we call \(\thdtt_{\tth}\) the
  \emph{essentially algebraic theory of \(\tth\)-theories} or the
  \emph{essentially algebraic theory for \(\tth\)} because
  \(\tth\)-theories are equivalent to ``models of \(\thdtt_{\tth}\)'',
  that is, cartesian functors \(\thdtt_{\tth} \to \Set\). Precisely,
  the functor
  \begin{equation}
    \label{eq:31}
    \DTT_{\tth} \ni \gat \mapsto \DTT_{\tth}(\argu, \gat) \in
    \Cart(\thdtt_{\tth}, \Set)
  \end{equation}
  is an equivalence.
\end{remark}

\subsection{Exponentiable Arrows Associated to Type Theories}
\label{sec:expon-arrows-assoc}

In this subsection we show that, for every type theory \(\tth\), the
cartesian category \(\thdtt_{\tth}\) has a certain exponentiable arrow
(\cref{cor:gat-exp}). We fix a type theory \(\tth\) and call a
\(\tth\)-theory simply a theory. We also omit the subscript
\({}_{\tth}\) of \(\thdtt_{\tth}\).

\begin{construction}
  \label{def:dtt-family}
  We define \(\Ty_{\nat}\) to be the theory consisting of the
  following symbols
  \begin{align*}
    \ty_{0}
    &: () \To \s{Type} \\
    \ty_{1}
    &: (\var_{0} : \ty_{0}) \To \s{Type} \\
    \ty_{2}
    &: (\var_{0} : \ty_{0}, \var_{1} : \ty_{1}) \To \s{Type} \\
    \vdots \\
    \ty_{\nat}
    &: (\var_{0} : \ty_{0}, \dots, \var_{\nat-1} : \ty_{\nat-1}) \To
      \s{Type}
  \end{align*}
  and \(\El_{\nat}\) to be the theory consisting of the symbols
  \(\ty_{0}, \dots, \ty_{\nat}\) of \(\Ty_{\nat}\) and a symbol
  \[
    \el_{\nat} : (\var_{0} : \ty_{0}, \dots, \var_{\nat-1} :
    \ty_{\nat-1}) \To \ty_{\nat}.
  \]
  We denote by \(\typeof_{\nat} : \El_{\nat} \to \Ty_{\nat}\) and
  \(\ctxof_{\nat} : \Ty_{\nat+1} \to \Ty_{\nat}\) the arrows in
  \(\thdtt\) represented by the obvious inclusions
  \(\Ty_{\nat} \to \El_{\nat}\) and \(\Ty_{\nat} \to \Ty_{\nat+1}\)
  respectively.
\end{construction}

We will show that the arrow \(\typeof_{0} : \El_{0} \to \Ty_{0}\) in
\(\thdtt\) is exponentiable (\cref{cor:gat-exp}). Although it is
possible to construct the pushforward \((\typeof_{0})_{*}\) directly,
we use \cref{thm:family-repr} to emphasize the connection between the
exponentiability of \(\typeof_{0}\) and the structural rules of
weakening, projection and substitution. The idea is that the stability
under hypothesizing defines an endofunctor and that the weakening,
projection and substitution rules correspond to the data \(\wk\),
\(\proj\) and \(\subst\), respectively, of a polynomial functor.

\begin{construction}
  We define a functor \(\poly : \thdtt \to \thdtt\) as follows:
  \begin{itemize}
  \item we reserve a symbol \(\ty_{0}\) and define
    \(\poly\gat = \gat^{\ty_{0}}\), that is, \(\poly\) sends a finite
    theory
    \begin{align*}
      \cst_{1}
      &: \ctx_{1} \To \kind_{1} \\
      \vdots \\
      \cst_{\nat}
      &: \ctx_{\nat} \To \kind_{\nat}
    \end{align*}
    to the finite theory
    \begin{align*}
      \ty_{0}
      &: () \To \s{Type} \\
      \hypsz{\var_{0}}\cst_{1}
      &: (\var_{0} : \ty_{0}, \hypsz{\var_{0}}\ctx_{1}) \To \hypsz{\var_{0}}\kind_{1} \\
      \vdots \\
      \hypsz{\var_{0}}\cst_{\nat}
      &: (\var_{0} : \ty_{0}, \hypsz{\var_{0}}\ctx_{\nat}) \To \hypsz{\var_{0}}\kind_{\nat}
    \end{align*}
    which is indeed a theory by the stability under hypothesizing;
  \item \(\poly\) sends an interpretation
    \(\itpr : \gat_{1} \to \gat_{2}\) to the interpretation
    \(\poly\itpr\) defined to be identity on \(\{\ty_{0}\}\) and
    \((\var_{0})_{!}\itpr\) on \((\var_{0})_{!}\set_{\gat_{1}}\). By
    the commutativity of \cref{eq:29}, we have
    \((\poly\itpr)^{\Kl}(\hypsz{\var_{0}}\expr) =
    \hypsz{\var_{0}}\itpr^{\Kl}(\expr)\) for any
    \(\expr \in \thexp(\set_{\gat_{1}}, \seq{\var})\), from which it
    follows that \(\poly\itpr\) is indeed an interpretation: for any
    \((\cst : \ctx \To \kind) \in \gat_{1}\), we have
    \(\var_{0} : \ty_{0}, \hypsz{\var_{0}}\itpr^{\Kl}(\ctx) \vdash
    \hypsz{\var_{0}}\itpr(\cst) : \hypsz{\var_{0}}\itpr^{\Kl}(\kind)\)
    by the stability under hypothesizing, and then rewrite it along
    the equation
    \(\hypsz{\var_{0}}\itpr^{\Kl}(\expr) =
    (\poly\itpr)^{\Kl}(\hypsz{\var_{0}}\expr)\).
  \end{itemize}
  The obvious inclusion \(\Ty_{0} \to \poly\gat\) defines a natural
  transformation \(\leg_{\gat} : \poly\gat \to \Ty_{0}\). Clearly
  \(\poly\) sends finite limits in \(\thdtt\) to finite limits in
  \(\thdtt/\Ty_{0}\), and thus \(\poly : \thdtt \to \thdtt\) preserves
  pullbacks.
\end{construction}

Intuitively, \(\poly\gat\) is the theory obtained from \(\gat\) by
\begin{enumerate}
\item adjoining a fresh type symbol \(\ty_{0}\);
\item modifying the arity \(\seq{\var}\) of each element \(\cst \in
  \gat\) to \((\var_{0}, \seq{\var})\) where \(\var_{0}\) is a
  variable of type \(\ty_{0}\).
\end{enumerate}
Then elements from \(\gat\) become indexed over \(\ty_{0}\) in
\(\poly\gat\), so the theory \(\poly\gat\) is considered the
\emph{theory of families of \(\gat\)}.

\begin{example}
  Let \(\gat\) be the theory of categories
  (\cref{exm:gat-of-cat}). Then \(\poly\gat\) is isomorphic to the
  following theory.
  \begin{align*}
    \ty_{0}
    &: () \To \s{Type} \\
    \thcatOb
    &: (\var_{0} : \ty_{0}) \To \s{Type} \\
    \thcatHom
    &: (\var_{0} : \ty_{0}, \var_{1} : \thcatOb(\var_{0}), \var_{2} :
      \thcatOb(\var_{0})) \To \s{Type} \\
    \thcatid
    &: (\var_{0} : \ty_{0}, \var : \thcatOb(\var_{0})) \To
      \thcatHom(\var_{0}, \var, \var) \\
    \thcatcomp
    & : (\var_{0} : \ty_{0}, \var_{1} : \thcatOb(\var_{0}),
      \var_{2} : \thcatOb(\var_{0}), \var_{3} :
      \thcatOb(\var_{0}), \\
    &\quad \varI_{1} :
      \thcatHom(\var_{0}, \var_{1}, \var_{2}), \varI_{2} :
      \thcatHom(\var_{0}, \var_{2}, \var_{3})) \To
      \thcatHom(\var_{0}, \var_{1}, \var_{3})
    \\
    \_
    &: (\var_{0}, : \ty_{0}, \var_{1} : \thcatOb(\var_{0}),
      \var_{2} : \thcatOb(\var_{0}), \varI :
      \thcatHom(\var_{0}, \var_{1}, \var_{2})) \\
    &\quad \To \thcatcomp(\var_{0}, \var_{1}, \var_{1},
      \var_{2}, \thcatid(\var_{0}, \var_{1}), \varI) = \varI
    \\
    \_
    &: (\var_{0} : \ty_{0}, \var_{1} : \thcatOb(\var_{0}),
      \var_{2} : \thcatOb(\var_{0}), \varI :
      \thcatHom(\var_{0}, \var_{1}, \var_{2})) \\
    &\quad \To \thcatcomp(\var_{0}, \var_{1}, \var_{2},
      \var_{2}, \varI, \thcatid(\var_{0}, \var_{2})) = \varI
    \\
    \_
    &: (\var_{0} : \ty_{0}, \var_{1} : \thcatOb(\var_{0}),
      \var_{2} : \thcatOb(\var_{0}), \var_{3} :
      \thcatOb(\var_{0}), \var_{4} :
      \thcatOb(\var_{0}), \\
    &\quad \varI_{1} :
      \thcatHom(\var_{0}, \var_{1}, \var_{2}), \varI_{2} :
      \thcatHom(\var_{0}, \var_{2}, \var_{3}), \varI_{3} :
      \thcatHom(\var_{0}, \var_{3}, \var_{4})) \To \\
    &\quad \thcatcomp(\var_{0}, \var_{1}, \var_{3},
      \var_{4}, \thcatcomp(\var_{0}, \var_{1}, \var_{2},
      \var_{3}, \varI_{1}, \varI_{2}), \varI_{3}) \\
    &\quad =
      \thcatcomp(\var_{0}, \var_{1}, \var_{2}, \var_{4},
      \varI_{1}, \thcatcomp(\var_{0}, \var_{2}, \var_{3},
      \var_{4}, \varI_{2}, \varI_{3}))
  \end{align*}
  One can think of \((\thcatOb, \thcatHom, \thcatid, \thcatcomp)\) as
  a family of categories indexed over \(\ty_{0}\), so \(\poly\gat\) is
  the theory of families of categories.
\end{example}

\begin{example}
  \label{exm:family-theory}
  We have isomorphisms
  \begin{gather*}
    \poly\Ty_{\nat} \cong \Ty_{\nat+1} \\
    \poly\El_{\nat} \cong \El_{\nat+1}.
  \end{gather*}
\end{example}

We will turn the functor \(\poly : \thdtt \to \thdtt\) into a
polynomial functor for \(\typeof_{0}\). We extensively use the
presentation of a pullback in \(\thdtt\) (pushout in \(\DTT\)) given
in \cref{prop:theories-pushouts}. For example, for an arrow
\(\itpr : \gat \to \Ty_{0}\), which is an interpretation
\(\Ty_{0} \to \gat\), the pullback of \(\El_{0}\) along \(\itpr\) in
\(\thdtt\) is isomorphic to the following theory.
\begin{align*}
  \cst
  &: \ctx \To \kind
  & ((\cst : \ctx \To \kind) \in \gat) \\
  \el_{0}
  &: () \To \itpr(\ty_{0})
\end{align*}

\begin{construction}
  \label{def:dtt-wk}
  For a finite theory \(\gat\), we define an arrow
  \(\wk_{\gat} : \gat \times \Ty_{0} \to \poly\gat\) over \(\Ty_{0}\)
  to be the interpretation \(\poly\gat \to \gat \times \Ty_{0}\)
  defined by
  \begin{align*}
    \ty_{0}
    &\mapsto ({} \vdash \ty_{0} : \s{Type}) \\
    (\hypsz{\var_{0}}\cst : (\var_{0} : \ty_{0}, \hypsz{\var_{0}}\ctx) \To
    \hypsz{\var_{0}}\kind)
    &\mapsto (\var_{0} : \ty_{0}, \ctx \vdash \cst : \kind)
    & (\cst \in \gat).
  \end{align*}
  This interpretation is well-defined because of the weakening
  rule. Clearly \(\wk\) is natural in \(\gat \in \thdtt\).
\end{construction}

\begin{construction}
  \label{def:dtt-proj}
  We define an arrow
  \(\proj : \Ty_{0} \to \poly\El_{0} \cong \El_{1}\) over \(\Ty_{0}\)
  to be the interpretation \(\El_{1} \to \Ty_{0}\) defined by
  \begin{align*}
    \ty_{0}
    &\mapsto ({} \vdash \ty_{0} : \s{Type}) \\
    \ty_{1}
    &\mapsto (\var_{0} : \ty_{0} \vdash \ty_{0} : \s{Type}) \\
    \el_{1}
    &\mapsto (\var_{0} : \ty_{0} \vdash \var_{0} : \ty_{0}).
  \end{align*}
  This interpretation is well-defined because of the projection rule.
\end{construction}

\begin{construction}
  \label{def:dtt-subst}
  For a finite theory \(\gat\), we define an arrow
  \(\subst_{\gat} : \poly\gat \times_{\Ty_{0}} \El_{0} \to \gat\) to
  be the interpretation \(\gat \to \poly\gat \times_{\Ty_{0}}
  \El_{0}\) defined by
  \begin{align*}
    (\cst : \ctx \To \kind)
    &\mapsto ((\hypsz{\var_{0}}\ctx)[\el_{0}/\var_{0}]
      \vdash (\hypsz{\var_{0}}\cst)[\el_{0}/\var_{0}] :
      (\hypsz{\var_{0}}\kind)[\el_{0}/\var_{0}])
    & (\cst \in \gat).
  \end{align*}
  This interpretation is well-defined because of the substitution
  rule. Clearly \(\subst\) is natural in \(\gat \in \thdtt\).
\end{construction}

\begin{proposition}
  \label{prop:dtt-family-structure}
  \((\poly, \leg, \wk, \proj, \subst)\) defined above is a
  polynomial functor for \(\typeof_{0}\).
\end{proposition}
\begin{proof}
  It remains to check \cref{axiom:1,axiom:2,axiom:3,axiom:4} of
  \cref{def:family-structure}. Intuitively, these axioms express the
  interaction of the weakening, projection and substitution
  operations.
  \begin{itemize}
  \item \Cref{axiom:1} expresses that the type of the projection
    \(\var_{0} : \ty_{0} \vdash \var_{0} : \ty_{0}\) is the weakening
    of \(\ty_{0}\) by \((\var_{0} : \ty_{0})\).
  \item \Cref{axiom:2} expresses that the substitution of \(\el_{0}\) for
    \(\var_{0}\) in \(\var_{0}\) is \(\el_{0}\):
    \({\var_{0}}[\el_{0}/\var_{0}] = \el_{0}\).
  \item \Cref{axiom:3} expresses that substitution in weakening is the
    identity: \({\expr}[\el_{0}/\var_{0}] = \expr\) for any expression
    \(\expr\) that does not contain \(\var_{0}\) as a free variable.
  \item \Cref{axiom:4} expresses that renaming of a variable is
    invertible:
    \({\expr}[\var_{1}/\var_{0}][\var_{0}/\var_{1}] = \expr\) for any
    expression \(\expr\) that does not contain \(\var_{1}\) as a free
    variable.
  \end{itemize}

  \paragraph{\Cref{axiom:1}} The arrows
  \(\poly\typeof_{0} \circ \proj : \Ty_{0} \to \poly\Ty_{0} \cong
  \Ty_{1}\) and
  \(\wk_{\Ty_{0}} \circ (\id_{\Ty_{0}}, \id_{\Ty_{0}}) : \Ty_{0} \to
  \poly\Ty_{0} \cong \Ty_{1}\) are both represented by the
  interpretation \(\Ty_{1} \to \Ty_{0}\) defined by
  \begin{align*}
    \ty_{0}
    &\mapsto ({} \vdash \ty_{0} : \s{Type}) \\
    \ty_{1}
    &\mapsto (\var_{0} : \ty_{0} \vdash \ty_{0} : \s{Type}).
  \end{align*}

  \paragraph{\Cref{axiom:2}}
  \(\poly\El_{0} \times_{\Ty_{0}} \El_{0}\) is isomorphic to the
  theory
  \begin{align*}
    \ty_{0}
    &: () \To \s{Type} \\
    \el_{0}
    &: () \To \ty_{0} \\
    \tyI_{1}
    &: (\var_{0} : \ty_{0}) \To \s{Type} \\
    \elI_{1}
    &: (\var_{0} : \ty_{0}) \To \tyI_{1}(\var_{0}).
  \end{align*}
  Then
  \(\subst_{\El_{0}} : \poly\El_{0} \times_{\Ty_{0}} \El_{0} \to
  \El_{0}\) and
  \(\proj \times_{\Ty_{0}} \El_{0} : \El_{0} \cong \Ty_{0}
  \times_{\Ty_{0}} \El_{0} \to \poly\El_{0} \times_{\Ty_{0}} \El_{0}\)
  are represented by the interpretation \(\El_{0} \to \poly\El_{0}
  \times_{\Ty_{0}} \El_{0}\) defined by
  \begin{align*}
    \ty_{0}
    &\mapsto ({} \vdash \tyI_{1}(\el_{0}) : \s{Type}) \\
    \el_{0}
    &\mapsto ({} \vdash \elI_{1}(\el_{0}) : \tyI_{1}(\el_{0}))
  \end{align*}
  and the interpretation \(\poly\El_{0} \times_{\Ty_{0}} \El_{0} \to
  \El_{0}\) defined by
  \begin{align*}
    \ty_{0}
    &\mapsto ({} \vdash \ty_{0} : \s{Type}) \\
    \el_{0}
    &\mapsto ({} \vdash \el_{0} : \ty_{0}) \\
    \tyI_{1}
    &\mapsto (\var_{0} : \ty_{0} \vdash \ty_{0}) \\
    \elI_{1}
    &\mapsto (\var_{0} : \ty_{0} \vdash \var_{0} : \ty_{0})
  \end{align*}
  respectively. Since \(\var_{0}[\el_{0}/\var_{0}] = \el_{0}\), the
  composite
  \(\subst_{\El_{0}} \circ (\proj \times_{\Ty_{0}} \El_{0})\) is
  represented by the identity interpretation.

  \paragraph{\Cref{axiom:3}} Since
  \({\expr}[\el_{0}/\var_{0}] = \expr\) for any expression that does
  not contain \(\var_{0}\) as a free variable, one can see that the
  composite
  \(\subst_{\gat} \circ (\wk_{\gat} \times_{\Ty_{0}} \El_{0}) : \gat
  \times \El_{0} \to \gat\) is represented by the inclusion
  \(\gat \to \gat \times\El_{0}\).

  \paragraph{\Cref{axiom:4}} Recall that
  \(\poly\gat \times_{\Ty_{0}} \El_{0}\) is isomorphic to the theory
  \begin{align*}
    \ty_{0}
    &: () \To \s{Type} \\
    \el_{0}
    &: () \To \ty_{0} \\
    \hypsz{\var_{0}}\cst
    &: (\var_{0} : \ty_{0}, \hypsz{\var_{0}}\ctx) \To \hypsz{\var_{0}}\kind
    & (\cst \in \gat).
  \end{align*}
  Thus, \(\poly(\poly\gat \times_{\Ty_{0}} \El_{0})\) is isomorphic to
  the theory
  \begin{align*}
    \ty_{0}
    &: () \To \s{Type} \\
    \ty_{1}
    &: (\var_{0} : \ty_{0}) \To \s{Type} \\
    \el_{1}
    &: (\var_{0} : \ty_{0}) \To \ty_{1} \\
    \hypsz{\var_{0}, \var_{1}}\cst
    &: (\var_{0} : \ty_{0}, \var_{1} : \ty_{1},
      \hypsz{\var_{0}, \var_{1}}\ctx) \To \hypsz{\var_{0}, \var_{1}}\kind
    & (\cst \in \gat)
  \end{align*}
  where \(\hypsz{\var_{0}, \var_{1}}\) is constructed from the
  adjunction \((\var_{0})_{!}^{2} \adj (\var_{0}^{*})^{2}\) in the
  same way as \(\hypsz{\var_{0}}\). Then,
  \((\wk_{\poly\gat}(\id_{\poly\gat}, \leg_{\gat}), \proj\leg_{\gat})
  : \poly\gat \to \poly(\poly\gat \times_{\Ty_{0}} \El_{0})\) and
  \(\poly\subst_{\gat} : \poly(\poly\gat \times_{\Ty_{0}} \El_{0}) \to
  \poly\gat\) are represented by the interpretation \(\poly(\poly\gat
  \times_{\Ty_{0}} \El_{0}) \to \poly\gat\) defined by
  \begin{align*}
    \ty_{0}
    &\mapsto ({} \vdash \ty_{0} : \s{Type}) \\
    \ty_{1}
    &\mapsto (\var_{0} : \ty_{0} \vdash \ty_{0} : \s{Type}) \\
    \el_{1}
    &\mapsto (\var_{0} : \ty_{0} \vdash \var_{0} : \ty_{0}) \\
    \hypsz{\var_{0}, \var_{1}}\cst
    &\mapsto (\var_{0} : \ty_{0}, \var_{1} : \ty_{0},
      (\hypsz{\var_{0}}\ctx)[\var_{1}/\var_{0}] \vdash \\
    &\hphantom{\mapsto}\quad (\hypsz{\var_{0}}\cst)[\var_{1}/\var_{0}] :
      (\hypsz{\var_{0}}\kind)[\var_{1}/\var_{0}])
    & (\cst \in \gat)
  \end{align*}
  and the interpretation \(\poly\gat \to \poly(\poly\gat
  \times_{\Ty_{0}} \El_{0})\) defined by
  \begin{align*}
    \ty_{0}
    &\mapsto ({} \vdash \ty_{0} : \s{Type}) \\
    \hypsz{\var_{0}}\cst
    &\mapsto (\var_{0} : \ty_{0},
      (\hypsz{\var_{0}, \var_{1}}\ctx)[\el_{1}(\var_{0})/\var_{1}]
      \vdash \\
    &\hphantom{\mapsto}\quad (\hypsz{\var_{0},
      \var_{1}}\cst)[\el_{1}(\var_{0})/\var_{1}] : (\hypsz{\var_{0},
      \var_{1}}\kind)[\el_{1}(\var_{0})/\var_{1}])
    & (\cst \in \gat)
  \end{align*}
  respectively. Since
  \({\expr}[\var_{1}/\var_{0}][\var_{0}/\var_{1}] = \expr\) for any
  expression \(\expr\) that does not contain \(\var_{1}\) as a free
  variable, the composite
  \(\poly\subst_{\gat} \circ (\wk_{\poly\gat}(\id_{\poly\gat},
  \leg_{\gat}), \proj\leg_{\gat})\) is represented by the identity
  interpretation.
\end{proof}

\begin{corollary}
  \label{cor:gat-exp}
  The arrow \(\typeof_{0} : \El_{0} \to \Ty_{0}\) in
  \(\thdtt\) is exponentiable.
\end{corollary}

\begin{corollary}
  \(\Ty_{\nat} \cong \poly_{\typeof_{0}}^{\nat}\Ty_{0}\) and
  \(\El_{\nat} \cong \poly_{\typeof_{0}}^{\nat}\El_{0}\).
\end{corollary}
\begin{proof}
  By \cref{exm:family-theory}.
\end{proof}

\section{The Universal Exponentiable Arrow}
\label{sec:univ-expon-arrow}

In this section we restrict our attention to the type theory
\(\tth_{0}\) without any type constructors (\cref{exm:tth-0}). We can
identify the \(\tth_{0}\)-theories as the generalized algebraic
theories (\cref{exm:t0-theory}), so we write \(\GAT\) for the category
\(\DTT_{\tth_{0}}\). Then \(\thgat := \thdtt_{\tth_{0}}\) is the
cartesian category such that \(\GAT \simeq \Cart(\thgat, \Set)\)
(\cref{rem:locally-pres}). By \cref{cor:gat-exp} the arrow
\(\typeof_{0} : \El_{0} \to \Ty_{0}\) in the category \(\thgat\) is
exponentiable, but this time we can say more: \(\typeof_{0}\) is the
\emph{universal exponentiable arrow} in the following sense.
\begin{theorem}
  \label{thm:gat-ump}
  For any cartesian category \(\cart\) and exponentiable arrow
  \(\typeof : \El \to \Ty\) in \(\cart\), there exists a unique, up to
  unique isomorphism, cartesian functor \(\functor : \thgat \to \cart\)
  such that \(\typeof \cong \functor\typeof_{0}\) and \(\functor\) sends
  pushforwards along \(\typeof_{0}\) to those along \(\typeof\).
\end{theorem}

This section is mostly devoted to proving \cref{thm:gat-ump}. In
\cref{sec:constr-funct-funct}, we construct a functor
\(\functor : \thgat \to \cart\) satisfying the required conditions. In
\cref{sec:uniqueness-functor-}, we show the uniqueness of such a
functor. We discuss variants of \cref{thm:gat-ump} in
\cref{sec:generalization}.

Throughout the proof of \cref{thm:gat-ump}, we extensively use the
following inductive presentations of finite generalized algebraic
theories. By definition, a finite generalized algebraic theory
\(\gat\) is a finite set of symbols and axioms with a well-founded
relation. We can arrange them into a list
\begin{align*}
  \cst_{1}
  &: \ctx_{1} \To \kind_{1} \\
  \cst_{2}
  &: \ctx_{2} \To \kind_{2} \\
  \vdots \\
  \cst_{\nat}
  &: \ctx_{\nat} \To \kind_{\nat}
\end{align*}
such that
\(\downset{\cst_{\idx}} \subset \{\cst_{1}, \dots, \cst_{\idx-1}\}\)
for every \(\idx\) so that every prefix
\((\cst_{1}, \dots, \cst_{\idx})\) is again a generalized algebraic
theory. Then the following inductive clauses cover all the objects of
\(\thgat\):
\begin{itemize}
\item \(\thgat\) has a terminal object;
\item if \(\gat\) is an object of \(\thgat\) and \(\ctx\) is a context
  of length \(\nat\), which corresponds to an arrow
  \(\gat \to \Ty_{\nat - 1}\), then we have an object
  \(\gat' = (\gat, \ty : \ctx \To \s{Type})\) fitting into the
  pullback
  \[
    \begin{tikzcd}
      \gat'
      \arrow[r]
      \arrow[d] &
      \Ty_{\nat}
      \arrow[d,"\ctxof_{\nat-1}"] \\
      \gat
      \arrow[r,"\ctx"'] &
      \Ty_{\nat-1}
    \end{tikzcd}
  \]
  in \(\thgat\), where we set \(\Ty_{-1}\) to be the terminal object;
\item if \(\gat\) is an object of \(\thgat\) and \(\ty\) is a type
  over a context of length \(\nat\), which corresponds to an arrow
  \(\gat \to \Ty_{\nat}\), then we have an object
  \(\gat' = (\gat, \el : \ctx \To \ty)\) fitting into the pullback
  \[
    \begin{tikzcd}
      \gat'
      \arrow[r]
      \arrow[d] &
      \El_{\nat}
      \arrow[d,"\typeof_{\nat}"] \\
      \gat
      \arrow[r,"\ty"'] &
      \Ty_{\nat}
    \end{tikzcd}
  \]
  in \(\thgat\);
\item if \(\gat\) is an object of \(\thgat\) and \(\ty_{1}\) and
  \(\ty_{2}\) are types over the same context \(\ctx\), then we have
  an object
  \(\gat' = (\gat, \_ : \ctx \To \ty_{1} = \ty_{2} : \s{Type})\)
  fitting into the equalizer
  \[
    \begin{tikzcd}
      \gat'
      \arrow[r] &
      \gat
      \arrow[r,shift left,"\ty_{1}"]
      \arrow[r,shift right,"\ty_{2}"'] &
      \Ty_{\nat}
    \end{tikzcd}
  \]
  in \(\thgat\);
\item if \(\gat\) is an object of \(\thgat\) and \(\el_{1}\) and
  \(\el_{2}\) are elements of the same type \(\ty\), which correspond
  to arrows \(\gat \to \El_{\nat}\), then we have an object
  \(\gat' = (\gat, \_ : \ctx \To \el_{1} = \el_{2} : \ty)\) fitting
  into the equalizer
  \[
    \begin{tikzcd}
      \gat'
      \arrow[r] &
      \gat
      \arrow[r,shift left,"\el_{1}"]
      \arrow[r,shift right,"\el_{2}"'] &
      \El_{\nat}
    \end{tikzcd}
  \]
  in \(\thgat\).
\end{itemize}
From this presentation, we get the following ``induction principle''.

\begin{lemma}
  \label{prop:gat-induction}
  Let \(\pred\) be a predicate on objects of \(\thgat\).
  \begin{enumerate}
  \item If all \(\Ty_{\nat}\) and \(\El_{\nat}\) satisfy \(\pred\) and
    \(\pred\) is stable under finite limits (that is, for a finite
    diagram \(\gat : \idxcat \to \thgat\), if every \(\gat_{\idx}\)
    satisfies \(\pred\) then the limit
    \(\lim_{\idx \in \idxcat}\gat_{\idx}\) satisfies \(\pred\)), then
    all the objects of \(\thgat\) satisfy \(\pred\).
  \item If \(\Ty_{0}\) and \(\El_{0}\) satisfy \(\pred\) and \(\pred\)
    is stable under finite limits and stable under
    \(\poly_{\typeof_{0}}\) (that is, if \(\gat\) satisfies \(\pred\),
    then \(\poly_{\typeof_{0}}\gat\) satisfies \(\pred\)), then all
    the objects of \(\thgat\) satisfy \(\pred\).
  \end{enumerate}
  \qed
\end{lemma}

\subsection{Constructing a functor \(\functor : \thgat \to \cart\)}
\label{sec:constr-funct-funct}

Let us fix a cartesian category \(\cart\) and an exponentiable arrow
\(\typeof : \El \to \Ty\) in \(\cart\). We construct a cartesian functor
\(\functor : \thgat \to \cart\) that sends \(\typeof_{0}\) to \(\typeof\)
and pushforwards along \(\typeof_{0}\) to those along \(\typeof\). The
outline is as follows.
\begin{enumerate}
\item A cartesian functor \(\thgat \to \cart\) is thought of as an
  ``internal generalized algebraic theory in \(\cart\)'', so it would
  have an ``externalization'' \(\cart^{\op} \to \GAT\). The
  externalization is easier to describe than the internal generalized
  algebraic theory itself, so we will first define a functor
  \(\iL : \cart^{\op} \to \GAT\).
\item Using the equivalence \(\GAT \simeq \Cart(\thgat, \Set)\), the
  functor \(\iL\) corresponds to a cartesian functor
  \(\widetilde{\iL} : \thgat \to [\cart^{\op}, \Set]\). We show that
  the functor \(\widetilde{\iL}\) factors through the Yoneda
  embedding, that is, \(\widetilde{\iL}\gat\) is representable for
  every \(\gat \in \thgat\). Let \(\functor : \thgat \to \cart\) be
  the induced functor.
\item We show that the functor \(\functor : \thgat \to \cart\)
  satisfies the required conditions.
\end{enumerate}

In what follows, for an arrow
\(\ty : \obj \to \poly_{\typeof}^{\nat}\Ty\) and \(\idx \le \nat\), we
refer to the composite
\(\obj \overset{\ty}{\longrightarrow} \poly_{\typeof}^{\nat}\Ty \to
\poly_{\typeof}^{\idx}\Ty\) as \(\ty_{\idx}\).

\begin{construction}
  We define a functor
  \[
    \iL : \cart^{\op} \to \GAT
  \]
  as follows. For an object \(\obj \in \cart\), we define \(\iL\obj\)
  to be the generalized algebraic theory consisting of the following
  data:
  \begin{itemize}
  \item a symbol
    \[
      \ty : (\var_{0} : \ty_{0}, \dots, \var_{\nat-1} :
      \ty_{\nat-1}(\var_{0}, \dots, \var_{\nat-2})) \To \s{Type}
    \]
    for any arrow \(\ty : \obj \to \poly_{\typeof}^{\nat}\Ty\);
  \item a symbol
    \[
      \el : (\var_{0} : \ty_{0}, \dots, \var_{\nat-1} :
      \ty_{\nat-1}(\var_{0}, \dots, \var_{\nat-2})) \To \ty(\var_{0},
      \dots, \var_{\nat-1})
    \]
    for any arrow \(\el : \obj \to \poly_{\typeof}^{\nat}\El\), where
    \(\ty\) is the composite
    \(\obj \overset{\el}{\longrightarrow} \poly_{\typeof}^{\nat}\El \to
    \poly_{\typeof}^{\nat}\Ty\);
  \item an equation
    \begin{multline*}
      (\var_{0} : \ty_{0}, \dots,\var_{\natI-1} :
      \ty_{\natI-1}(\var_{0}, \dots, \var_{\natI-2}), \var'_{\natI} :
      \ty'_{\natI}(\var_{0}, \dots, \var_{\natI-1}), \\
      \var_{\natI} : \ty_{\natI}(\var_{0}, \dots, \var_{\natI-1}),
      \dots, \var_{\nat-1} : \ty_{\nat-1}(\var_{0}, \dots,
      \var_{\nat-2})) \To \\
      \cst(\var_{0}, \dots, \var_{\nat}) =
      (\poly_{\typeof}^{\natI}\wk_{\poly_{\typeof}^{\nat-\natI}\BE} \circ
      (\cst, \ty'_{\natI}))(\var_{0}, \dots, \var'_{\natI},
      \var_{\natI}, \dots, \var_{\nat})
    \end{multline*}
    for any arrow \(\cst : \obj \to \poly_{\typeof}^{\nat}\BE\) with
    \(\BE = \Ty\) or \(\BE = \El\) and any arrow
    \(\ty'_{\natI} : \obj \to \poly_{\typeof}^{\natI}\Ty\) over
    \(\ty_{\natI-1} : \obj \to \poly_{\typeof}^{\natI-1}\Ty\) with
    \(\natI \le \nat\);
  \item an equation
    \begin{multline*}
      (\var_{0} : \ty_{0}, \dots, \var_{\nat} : \ty_{\nat}(\var_{0},
      \dots, \var_{\nat-1})) \To \var_{\nat} =
      (\poly_{\typeof}^{\nat}\proj \circ \ty_{\nat})(\var_{0}, \dots,
      \var_{\nat})
    \end{multline*}
    for any arrow \(\ty : \obj \to \poly_{\typeof}^{\nat}\Ty\);
  \item an equation
    \begin{multline*}
      (\var_{0} : \ty_{0}, \dots, \var_{\natI-1} :
      \ty_{\natI-1}(\var_{0}, \dots, \var_{\natI-2}), \\
      \var_{\natI+1} : \ty_{\nat+1}(\var_{0}, \dots, \el_{\natI}),
      \dots, \var_{\nat} : \ty_{\nat}(\var_{0}, \dots, \el_{\natI},
      \dots, \var_{\nat-1})) \To \\
      \cst(\var_{0}, \dots, \el_{\natI}, \dots, \var_{\nat}) =
      (\poly_{\typeof}^{\natI}\subst_{\poly_{\typeof}^{\nat-\natI}\BE} \circ
      (\cst, \el_{\natI}))(\var_{0}, \dots, \var_{\natI-1},
      \var_{\natI+1}, \dots, \var_{\nat})
    \end{multline*}
    for any arrow \(\cst : \obj \to \poly_{\typeof}^{\nat+1}\BE\) with
    \(\BE = \Ty\) or \(\BE = \El\) and any arrow
    \(\el_{\natI} : \obj \to \poly_{\typeof}^{\natI}\El\) over
    \(\ty_{\natI} : \obj \to \poly_{\typeof}^{\natI}\Ty\) with
    \(\natI \le \nat\).
  \end{itemize}
  For an arrow \(\arr : \obj_{1} \to \obj_{2}\), the precomposition
  with \(\arr\) induces an interpretation
  \(\iL\arr : \iL\obj_{2} \to \iL\obj_{1}\).
\end{construction}

We have equivalences
\begin{align*}
  [\cart^{\op}, \GAT]
  &\simeq [\cart^{\op}, \Cart(\thgat, \Set)]
  & \text{(\cref{eq:31})} \\
  &\simeq \Cart(\thgat, [\cart^{\op}, \Set]).
\end{align*}
Let
\[
  \widetilde{\iL} : \thgat \to [\cart^{\op}, \Set]
\]
be the cartesian functor corresponding to
\(\iL : \cart^{\op} \to \GAT\). Concretely, one can define
\[
  \widetilde{\iL}\gat = (\cart^{\op} \ni \obj \mapsto \GAT(\gat,
  \iL\obj) \in \Set)
\]
for \(\gat \in \thgat\).

We show the following to obtain a functor
\(\functor : \thgat \to \cart\) satisfying the required conditions.

\begin{lemma}
  \label{item:1}
  \(\widetilde{\iL}\) factors through the Yoneda embedding
  \(\cart \to [\cart^{\op}, \Set]\) up to natural isomorphism. We
  write \(\functor : \thgat \to \cart\) for the induced functor which
  is automatically cartesian as the Yoneda embedding preserves finite
  limits.
  \[
    \begin{tikzcd}
      \thgat
      \arrow[r, dotted, "\functor"]
      \arrow[rd, "\widetilde{\iL}"'] &
      \cart
      \arrow[d, hook] \\
      & {[\cart^{\op}, \Set]}
    \end{tikzcd}
  \]
\end{lemma}
\begin{lemma}
  \label{item:2}
  \(\functor : \thgat \to \cart\) sends \(\typeof_{0}\) to \(\typeof\).
\end{lemma}
\begin{lemma}
  \label{item:3}
  \(\functor : \thgat \to \cart\) carries pushforwards along
  \(\typeof_{0}\) to those along \(\typeof\).
\end{lemma}
All of these follow from the following lemma.
\begin{lemma}
  \label{representability-lemma}
  \(\widetilde{\iL}\Ty_{\nat} : \cart^{\op} \to \Set\) is representable
  by \(\poly_{\typeof}^{\nat}\Ty\), and
  \(\widetilde{\iL}\El_{\nat} : \cart^{\op} \to \Set\) is representable
  by \(\poly_{\typeof}^{\nat}\El\).
\end{lemma}
\Cref{prop:gat-induction,representability-lemma} imply \cref{item:1},
because the predicate ``\(\widetilde{\iL}\gat\) is representable'' is
stable under finite limits as \(\widetilde{\iL}\) preserves finite
limits and \(\cart\) has finite limits. \Cref{item:2} is the special
case of \cref{representability-lemma} when \(\nat = 0\). For
\cref{item:3}, it suffices by \cref{prop:preserve-poly} to show that
the canonical natural transformation
\(\functor\poly_{\typeof_{0}} \To \poly_{\typeof}\functor\) is an
isomorphism. Since \(\functor : \thgat \to \cart\) preserves finite
limits and \(\poly_{\typeof_{0}}\) and \(\poly_{\typeof}\) preserve
finite limits as functors \(\thgat \to \thgat/\Ty_{0}\) and
\(\cart \to \cart/\Ty\) respectively, the predicate ``the canonical
arrow
\(\functor\poly_{\typeof_{0}}\gat \to \poly_{\typeof}\functor\gat\) is
an isomorphism'' is stable under finite limits. Hence, by
\cref{prop:gat-induction}, it suffices to show the cases of
\(\gat = \Ty_{\nat}\) and \(\gat = \El_{\nat}\). But this follows from
\cref{representability-lemma} because
\(\poly_{\typeof_{0}}\Ty_{\nat} \cong \Ty_{\nat+1}\) and
\(\poly_{\typeof_{0}}\El_{\nat} \cong \El_{\nat+1}\)
(\cref{exm:family-theory}).

The rest of the subsection is devoted to proving
\cref{representability-lemma}. Observe that
\(\widetilde{\iL}(\Ty_{\nat}, \obj) = \GAT(\Ty_{\nat}, \iL\obj)\) is
the set of equivalence classes of types in \(\iL\obj\) with \(\nat\)
variables and that
\(\widetilde{\iL}(\El_{\nat}, \obj) = \GAT(\El_{\nat}, \iL\obj)\) is
the set of equivalence classes of terms in \(\iL\obj\) with \(\nat\)
variables. We have the map
\[
  \inc_{\Ty} : \cart(\obj, \poly_{\typeof}^{\nat}\Ty) \to
  \widetilde{\iL}(\Ty_{\nat}, \obj)
\]
that sends an arrow \(\ty : \obj \to \poly_{\typeof}^{\nat}\Ty\) to the
type
\[
  \var_{0} : \ty_{0}, \dots, \var_{\nat-1} : \ty_{\nat-1}(\var_{0},
  \dots, \var_{\nat-2}) \vdash \ty(\var_{0}, \dots, \var_{\nat-1}) :
  \s{Type}
\]
and the map
\[
  \inc_{\El} : \cart(\obj, \poly_{\typeof}^{\nat}\El) \to
  \widetilde{\iL}(\El_{\nat}, \obj)
\]
that sends an arrow \(\el : \obj \to \poly_{\typeof}^{\nat}\El\) to the
term
\[
  \var_{0} : \ty_{0}, \dots, \var_{\nat-1} : \ty_{\nat-1}(\var_{0},
  \dots, \var_{\nat-2}) \vdash \el(\var_{0}, \dots, \var_{\nat-1}) :
  \ty(\var_{0}, \dots, \var_{\nat-1}).
\]
We show that the maps \(\inc_{\Ty}\) and \(\inc_{\El}\) are
bijective so that \(\widetilde{\iL}\Ty_{\nat}\) and
\(\widetilde{\iL}\El_{\nat}\) are representable by
\(\poly_{\typeof}^{\nat}\Ty\) and \(\poly_{\typeof}^{\nat}\El\)
respectively.

Inverses of \(\inc_{\Ty}\) and \(\inc_{\El}\) are defined by
interpreting the generalized algebraic theory \(\iL\obj\) in the
category with families associated to the arrow
\(\typeof : \El \to \Ty\). By a standard argument in the semantics of
dependent type theory \parencite[for example][]{hofmann1997syntax}, we
can define an interpretation \(\sem{\argu}\) of contexts, types and
terms of \(\iL\obj\) as follows:
\begin{itemize}
\item a context \(\ctx\) is interpreted as an object \(\sem{\ctx}
  \in \cart/\obj\);
\item a type \(\ctx \vdash \ty : \s{Type}\) is interpreted as an
  arrow \(\sem{\ty} : \sem{\ctx} \to \Ty\) in \(\cart\);
\item a term \(\ctx \vdash \el : \ty\) is interpreted as an arrow
  \(\sem{\el} : \sem{\ctx} \to \El\) over \(\sem{\ty}\);
\item the empty context \(()\) is interpreted as the terminal object
  \(\term_{\obj}\) in \(\cart/\obj\);
\item a context extension \(\ctx, \var : \ty\) is interpreted as the
  pullback
  \[
    \begin{tikzcd}
      \sem{\ctx, \var : \ty}
      \arrow[r]
      \arrow[d] &
      \El
      \arrow[d,"\typeof"] \\
      \sem{\ctx}
      \arrow[r,"\sem{\ty}"'] &
      \Ty;
    \end{tikzcd}
  \]
\item the type symbol \(\ty\) corresponding to an arrow
  \(\ty : \obj \to \poly_{\typeof}^{0}\Ty \cong \Ty\) is interpreted
  as \(\ty\) itself;
\item the type symbol \(\ty\) corresponding to an arrow
  \(\ty : \obj \to \poly_{\typeof}^{\nat+1}\Ty\) is interpreted as the
  arrow \(\sem{\ty} : \sem{\ty_{\nat}}^{*}\El \to \Ty\)
  corresponding to \(\ty\) via (iterated use of) the adjunction
  \(\typeof^{*} \adj \typeof_{*}\);
\item the term symbol \(\el\) corresponding to an arrow \(\el : \obj
  \to \poly_{\typeof}^{0}\El \cong \El\) is interpreted as \(\el\)
  itself;
\item the term symbol \(\el\) corresponding to an arrow
  \(\el : \obj \to \poly_{\typeof}^{\nat+1}\El\) is interpreted as the
  arrow \(\sem{\el} : \sem{\ty_{\nat}}^{*}\El \to \El\)
  corresponding to \(\el\) via (iterated use of) the adjunction
  \(\typeof^{*} \adj \typeof_{*}\);
\item the weakening \(\ctx, \var : \ty, \ctxI \vdash \judg\) of
  \(\ctx, \ctxI \vdash \judg\) is interpreted as the composite
  \[
    \begin{tikzcd}
      \sem{\ctx, \var : \ty, \ctxI}
      \arrow[r,"\cong"] &
      \sem{\ctx, \ctxI} \times_{\sem{\ctx}} \sem{\ctx, \var : \ty}
      \arrow[r,"\prodpr_{1}"] &
      \sem{\ctx, \ctxI}
      \arrow[r,"\sem{\judg}"] &
      \BE
    \end{tikzcd}
  \]
  where \(\BE = \Ty\) or \(\BE = \El\);
\item the projection \(\ctx, \var : \ty \vdash \var : \ty\) is
  interpreted as the second projection
  \(\sem{\ctx, \var : \ty} \cong \sem{\ctx} \times_{\Ty} \El
  \overset{\prodpr_{2}}{\longrightarrow} \El\);
\item the substitution \(\ctx, {\ctxI}[\el/\var] \vdash
  {\judg}[\el/\var]\) of \(\ctx \vdash \el : \ty\) is interpreted as
  the composite
  \[
    \begin{tikzcd}
      \sem{\ctx, {\ctxI}[\el/\var]}
      \arrow[r,"\overline{\sem{\el}}"] &
      \sem{\ctx, \var : \ty, \ctxI}
      \arrow[r,"\sem{\judg}"] &
      \BE
    \end{tikzcd}
  \]
  where \(\BE = \Ty\) or \(\BE = \El\) and
  \[
    \begin{tikzcd}
      \sem{\ctx, {\ctxI}[\el/\var]}
      \arrow[r,"\overline{\sem{\el}}"]
      \arrow[d] &
      \sem{\ctx, \var : \ty, \ctxI}
      \arrow[d] \\
      \sem{\ctx}
      \arrow[r,"{(\id_{\sem{\ctx}}, \sem{\el})}"'] &
      \sem{\ctx, \var : \ty}
    \end{tikzcd}
  \]
  is a pullback.
\end{itemize}
We have to verify that the interpretation \(\sem{\argu}\) satisfies
the equational axioms of \(\iL\obj\). First, for a type \(\ctx \vdash
\ty : \s{Type}\) and a term \(\ctx \vdash \el : \ty\) over a context
of length \(\nat\), we write
\[
  \inc_{\Ty, \ctx}^{-1}(\ty) : \obj \to \poly_{\typeof}^{\nat}\Ty
\]
for the arrow corresponding to \(\sem{\ty} : \sem{\ctx} \to \Ty\)
via the adjunction \(\typeof^{*} \adj \typeof_{*}\) and write
\[
  \inc_{\El, \ctx}^{-1}(\el) : \obj \to \poly_{\typeof}^{\nat}\El
\]
for the arrow corresponding to \(\sem{\el} : \sem{\ctx} \to
\El\). We omit the subscripts \(_{\Ty}\), \(_{\El}\) and
\(_{\ctx}\) when they are clear from the context. Then we have the
following equations from which it follows that \(\sem{\argu}\)
satisfies the axioms of \(\iL\obj\):
\begin{itemize}
\item for the symbol \(\cst\) corresponding to an arrow \(\obj \to
  \poly_{\typeof}^{\nat}\BE\),
  \begin{equation}
    \label{eq:14}
    \inc^{-1}(\cst(\var_{0}, \dots, \var_{\nat-1})) = \cst
  \end{equation}
\item for the weakening \(\ctx, \var : \ty, \ctxI \vdash \judg\) of
  \(\ctx, \ctxI \vdash \judg\),
  \begin{equation}
    \label{eq:11}
    \inc_{\ctx, \var : \ty, \ctxI}^{-1}(\judg) =
    \poly_{\typeof}^{\natI}\wk_{\poly_{\typeof}^{\nat-\natI}\BE} \circ
    (\inc_{\ctx, \ctxI}^{-1}(\judg), \inc_{\ctx}^{-1}(\ty))
  \end{equation}
\item for the projection \(\ctx, \var : \ty \vdash \var : \ty\),
  \begin{equation}
    \label{eq:12}
    \inc_{\ctx, \var : \ty}^{-1}(\var) = \poly_{\typeof}^{\nat}\proj
    \circ \inc_{\ctx}^{-1}(\ty)
  \end{equation}
\item for the substitution \(\ctx, {\ctxI}[\el/\var] \vdash
  {\judg}[\el/\var]\) of \(\ctx \vdash \el : \ty\),
  \begin{equation}
    \label{eq:13}
    \inc_{\ctx, \ctxI}^{-1}({\judg}[\el/\var]) =
    \poly_{\typeof}^{\natI}\subst_{\poly_{\typeof}^{\nat-\natI} \circ
      (\inc_{\ctx, \var : \el, \ctxI}^{-1}(\judg),
      \inc_{\ctx}^{-1}(\el))}
  \end{equation}
\end{itemize}
where \(\natI\) and \(\nat\) are the lengths of \(\ctx\) and
\((\ctx, \ctxI)\) respectively. \Cref{eq:14} is immediate from the
definition. \Cref{eq:11} follows from the following correspondence
via the adjunction \(\typeof^{*} \adj \typeof_{*}\).
\[
  \begin{array}{c}
    \begin{tikzcd}[ampersand replacement=\&]
      \sem{\ctx, \var : \ty, \ctxI}
      \arrow[r,"\prodpr_{1}"] \&
      \sem{\ctx, \ctxI}
      \arrow[r,"\sem{\judg}"] \&
      \BE
    \end{tikzcd} \\
    \hline
    \begin{tikzcd}[ampersand replacement=\&]
      \sem{\ctx, \var : \ty}
      \arrow[r,"\prodpr_{1}"] \&
      \sem{\ctx}
      \arrow[r,"\sem{\judg}'"] \&
      \poly_{\typeof}^{\nat-\natI}\BE
    \end{tikzcd} \\
    \hline
    \begin{tikzcd}[ampersand replacement=\&]
      \sem{\ctx} \times_{\Ty} \El
      \arrow[r,"{(\sem{\judg}', \sem{\ty}) \times_{\Ty} \El}"] \&
      [8ex]
      (\poly_{\typeof}^{\nat-\natI}\BE \times \Ty) \times_{\Ty} \El
      \arrow[r,"\prodpr_{1}"] \&
      \poly_{\typeof}^{\nat-\natI}\BE
    \end{tikzcd} \\
    \hline
    \begin{tikzcd}[ampersand replacement=\&]
      \sem{\ctx}
      \arrow[r,"{(\sem{\judg}', \sem{\ty})}"] \&
      [6ex]
      \poly_{\typeof}^{\nat-\natI}\BE \times \Ty
      \arrow[r,"\wk_{\poly_{\typeof}^{\nat-\natI}\BE}"] \&
      \poly_{\typeof}^{1+\nat-\natI}\BE
    \end{tikzcd} \\
    \hline
    \begin{tikzcd}[ampersand replacement=\&]
      \obj
      \arrow[r,"{(\inc^{-1}(\judg), \inc^{-1}(\ty))}"] \&
      [8ex]
      \poly_{\typeof}^{\natI}(\poly_{\typeof}^{\nat-\natI}\BE \times \Ty)
      \arrow[r,"\poly_{\typeof}^{\natI}\wk_{\poly_{\typeof}^{\nat-\natI}\BE}"]
      \&
      [4ex]
      \poly_{\typeof}^{1+\nat}\BE
    \end{tikzcd}
  \end{array}
\]
\Cref{eq:12,eq:13} are similar.

The assignments \(\ty \mapsto \inc_{\Ty}^{-1}(\ty)\) and \(\el
\mapsto \inc_{\El}^{-1}(\el)\) define maps
\begin{gather*}
  \inc_{\Ty}^{-1} : \widetilde{\iL}(\Ty_{\nat}, \obj) \to \cart(\obj,
  \poly_{\typeof}^{\nat}\Ty) \\
  \inc_{\El}^{-1} : \widetilde{\iL}(\El_{\nat}, \obj) \to \cart(\obj,
  \poly_{\typeof}^{\nat}\El)
\end{gather*}
respectively. By \cref{eq:14}, we have \(\inc^{-1} \circ \inc = \id\)
for \(\inc = \inc_{\Ty}\) and \(\inc = \inc_{\El}\). For the
equation \(\inc \circ \inc^{-1} = \id\), recall that types and terms
in \(\iL\obj\) are built up with type and term symbols of \(\iL\obj\)
using weakening, projection and substitution. Then, by induction on
derivation, one can show that every type or term \(\expr\) in
\(\iL\obj\) is provably equal to
\(\inc^{-1}(\expr)(\var_{0}, \dots, \var_{\nat-1})\). For example,
when \(\expr\) is the substitution \({\tyI}[\el/\var_{\natI}]\) of
\(\var_{0} : \ty_{0}, \dots, \var_{\natI-1} : \ty_{\natI-1} \vdash \el
: \ty_{\natI}\) for \(\var_{\natI}\) in
\(\var_{0} : \ty_{0}, \dots, \var_{\nat} : \ty_{\nat} \vdash \tyI :
\s{Type}\), we have
\begin{align*}
  & {\tyI}[\el/\var_{\natI}] \\
  = \tag{induction hypothesis} \\
  & \inc^{-1}(\tyI)(\var_{0}, \dots, \var_{\natI-1},
    \inc^{-1}(\el)(\var_{0}, \dots, \var_{\natI-1}), \var_{\natI+1},
    \dots, \var_{\nat}) \\
  = \tag{axom of \(\iL\obj\)} \\
  & (\poly_{\typeof}^{\natI}\subst_{\poly_{\typeof}^{\nat-\natI}\Ty} \circ
    (\inc^{-1}(\tyI), \inc^{-1}(\el)))(\var_{0}, \dots,
    \var_{\natI-1}, \var_{\natI+1}, \dots, \var_{\nat}) \\
  = \tag{\cref{eq:13}} \\
  & \inc^{-1}({\tyI}[\el/\var_{\natI}])(\var_{0}, \dots,
    \var_{\natI-1}, \var_{\natI+1}, \dots, \var_{\nat}).
\end{align*}
The other cases are similar. Hence, \(\inc_{\Ty}\) and
\(\inc_{\El}\) are bijective with inverses \(\inc_{\Ty}^{-1}\) and
\(\inc_{\El}^{-1}\), respectively, which completes the proof of
\cref{representability-lemma}.

\subsection{Uniqueness of \(\functor : \thgat \to \cart\)}
\label{sec:uniqueness-functor-}

For the uniqueness of a cartesian functor
\(\functor : \thgat \to \cart\) sending \(\typeof_{0}\) to \(\typeof\) and
pushforwards along \(\typeof_{0}\) to those along \(\typeof\), let
\(\functor' : \thgat \to \cart\) be another cartesian functor
satisfying the same conditions. We construct a natural isomorphism
\(\trans : \functor \cong \functor'\) such that the diagram
\[
  \begin{tikzcd}
    & \typeof
    \arrow[dl, "\cong"']
    \arrow[dr, "\cong"] \\
    \functor\typeof_{0}
    \arrow[rr, "\trans_{\typeof_{0}}"'] & &
    \functor'\typeof_{0}
  \end{tikzcd}
\]
commutes and show that such a natural isomorphism is unique. The idea
is to construct a natural isomorphism between the ``externalizations''
\(\iL \cong \iL' : \cart^{\op} \to \GAT\), and then the Yoneda Lemma
implies that it determines a natural isomorphism
\(\functor \cong \functor'\).

Let \(\iL' : \cart^{\op} \to \GAT\) be the functor corresponding to
the composite of \(\functor'\) and the Yoneda embedding
\(\cart \to [\cart^{\op}, \Set]\) via the equivalence
\([\cart^{\op}, \GAT] \simeq \Cart(\thgat, [\cart^{\op},
\Set])\). Concretely, \(\iL'\obj\) for \(\obj \in \cat\) is given by
the filtered colimit
\[
  \iL'\obj \cong \colim_{(\gat, \arr) \in (\obj \downarrow
    \functor')}\gat.
\]
We have a natural transformation \(\itpr : \iL \To
\iL' : \cart^{\op} \to \GAT\) defined as follows:
\begin{itemize}
\item for the type symbol \(\ty\) corresponding to an arrow
  \(\ty : \obj \to \poly_{\typeof}^{\nat}\Ty \cong \functor'\Ty_{\nat}\),
  we define \(\itpr_{\obj}(\ty)\) to be the type in \(\iL'\obj\)
  corresponding to the inclusion \(\Ty_{\nat} \to \iL'\obj\) at
  \((\Ty_{\nat}, \ty) \in (\obj \downarrow \functor')\);
\item for the term symbol \(\el\) corresponding to an arrow
  \(\el : \obj \to \poly_{\typeof}^{\nat}\El \cong \functor'\El_{\nat}\),
  we define \(\itpr_{\obj}(\el)\) to be the term in \(\iL'\obj\)
  corresponding to the inclusion \(\El_{\nat} \to \iL'\obj\) at
  \((\El_{\nat}, \el) \in (\obj \downarrow \functor')\).
\end{itemize}
By Yoneda, \(\itpr\) corresponds to a natural transformation
\(\trans : \functor \To \functor' : \thgat \to \cart\).

By definition, the diagrams
\begin{gather*}
  \begin{tikzcd}[ampersand replacement=\&]
    \cart(\obj, \Ty)
    \arrow[d,"\cong"']
    \arrow[r,"\cong"] \&
    [6ex]
    \cart(\obj, \functor'\Ty_{0})
    \arrow[d,"\cong"] \\
    \GAT(\Ty_{0}, \iL\obj)
    \arrow[r,"{\GAT(\Ty_{0}, \itpr_{\obj})}"'] \&
    \GAT(\Ty_{0}, \iL'\obj)
  \end{tikzcd}
  \\
  \begin{tikzcd}[ampersand replacement=\&]
    \cart(\obj, \El)
    \arrow[d,"\cong"']
    \arrow[r,"\cong"] \&
    [6ex]
    \cart(\obj, \functor'\El_{0})
    \arrow[d,"\cong"] \\
    \GAT(\El_{0}, \iL\obj)
    \arrow[r,"{\GAT(\El_{0}, \itpr_{\obj})}"'] \&
    \GAT(\El_{0}, \iL'\obj)
  \end{tikzcd}
\end{gather*}
commute for all objects \(\obj \in \cart\), and thus the diagrams
\begin{gather}
  \label[diagram]{eq:2}
  \begin{tikzcd}[ampersand replacement=\&]
    \& \Ty
    \arrow[dl,"\cong"']
    \arrow[dr,"\cong"] \\
    \functor\Ty_{0}
    \arrow[rr,"\trans_{\Ty_{0}}"'] \& \&
    \functor'\Ty_{0}
  \end{tikzcd}
  \\
  \label[diagram]{eq:3}
  \begin{tikzcd}[ampersand replacement=\&]
    \& \El
    \arrow[dl,"\cong"']
    \arrow[dr,"\cong"] \\
    \functor\El_{0}
    \arrow[rr,"\trans_{\El_{0}}"'] \& \&
    \functor'\El_{0}
  \end{tikzcd}
\end{gather}
commute. This also shows that \(\trans_{\Ty_{0}}\) and
\(\trans_{\El_{0}}\) are isomorphisms. Since \(\functor\) and
\(\functor'\) preserve finite limits and carry pushforwards along
\(\typeof_{0}\) to those along \(\typeof\), every component
\(\trans_{\gat}\) is an isomorphism by \cref{prop:gat-induction}.

It remains to show the uniqueness of such a natural isomorphism
\(\trans\). Let \(\trans' : \functor \To \functor'\) be another such
natural isomorphism. Then similar diagrams to \labelcref{eq:2,eq:3}
for \(\trans'\) commute, which implies that \(\trans'\) agrees with
\(\trans\) at \(\Ty_{0}\) and \(\El_{0}\). By
\cref{prop:gat-induction}, \(\trans'\) and \(\trans\) are equal.

We have constructed a cartesian functor \(\functor : \thgat \to \cart\)
sending \(\typeof_{0}\) to \(\typeof\) and pushforwards along
\(\typeof_{0}\) to those along \(\typeof\) and proved that such a functor
is unique up to unique isomorphism. This completes the proof of
\cref{thm:gat-ump}.

\subsection{Variants}
\label{sec:generalization}

We can obtain universal properties of \(\thdtt_{\tth}\) for another
type theory \(\tth\) in a similar way to \cref{thm:gat-ump}. Since a
finite \(\tth\)-theory is presented by a list of symbols and axioms,
\cref{prop:gat-induction} still holds for \(\thdtt_{\tth}\). Hence, a
minor modification of the proof of \cref{thm:gat-ump} works for a
variety of type theories.

\begin{example}
  Consider the type theory \(\tth_{\Pi}\) with \(\Pi\)-types
  (\cref{exm:tth-Pi}). The category \(\thdtt_{\tth_{\Pi}}\) contains a
  commutative diagram
  \begin{equation}
    \begin{tikzcd}
      \El_{1}
      \arrow[r,"\lambda"]
      \arrow[d,"\typeof_{1}"'] &
      \El_{0}
      \arrow[d,"\typeof_{0}"] \\
      \Ty_{1}
      \arrow[r,"\Pi"'] &
      \Ty_{0}
    \end{tikzcd}\label[diagram]{eq:1}
  \end{equation}
  where \(\Pi : \Ty_{1} \to \Ty_{0}\) is the arrow represented by the
  interpretation \(\Ty_{0} \to \Ty_{1}\) defined by
  \begin{align*}
    \ty_{0}
    &\mapsto ({} \vdash \textstyle{\prod_{\var_{0} :
      \ty_{0}}\ty_{1}(\var_{0})} : \s{Type})
  \end{align*}
  and \(\lambda : \El_{1} \to \El_{0}\) is the arrow represented by
  the interpretation \(\El_{0} \to \El_{1}\) defined by
  \begin{align*}
    \ty_{0}
    &\mapsto ({} \vdash \textstyle{\prod_{\var_{0} :
      \ty_{0}}\ty_{1}(\var_{0})} : \s{Type}) \\
    \el_{0}
    &\mapsto ({} \vdash \lambda\var_{0}.\el_{1}(\var_{0}) :
      \textstyle{\prod_{\var_{0} : \ty_{0}}\ty_{1}(\var_{0})}).
  \end{align*}
  The last three rules of \cref{fig:pi-types} force \cref{eq:1} to be
  a pullback. \(\thdtt_{\tth_{\Pi}}\), together with \cref{eq:1}, is
  universal in the following sense.
\end{example}

\begingroup
\def\PiI{P}
\def\lambdaI{l}

\begin{theorem}
  \label{thm:tth-pi-ump}
  For any cartesian category \(\cart\), exponentiable arrow
  \(\typeof : \El \to \Ty\) and pullback square of the form
  \begin{equation}
    \label[diagram]{eq:4}
    \begin{tikzcd}
      \poly_{\typeof}\El \arrow[r,"\lambdaI"] \arrow[d,"\poly_{\typeof}\typeof"']
      & \El
      \arrow[d,"\typeof"] \\
      \poly_{\typeof}\Ty \arrow[r,"\PiI"'] & \Ty,
    \end{tikzcd}
  \end{equation}
  there exists a unique, up to unique isomorphism, cartesian functor
  \(\functor : \thdtt_{\tth_{\Pi}} \to \cart\) such that
  \(\typeof \cong \functor\typeof_{0}\) and \(\functor\) sends
  pushforwards along \(\typeof_{0}\) to those along \(\typeof\) and
  \cref{eq:1} to \cref{eq:4}.
\end{theorem}

\begin{remark}
  \label{rem:natural-model}
  In the natural model semantics of dependent type theory
  \parencite{awodey2018natural,newstead2018thesis}, \(\Pi\)-types are
  modeled by a pullback square \labelcref{eq:4} in a presheaf
  category. Therefore, \cref{thm:tth-pi-ump} implies that a natural
  model with \(\Pi\)-types can be identified with a functor from
  \(\thdtt_{\tth_{\Pi}}\) to a presheaf category preserving finite
  limits and pushforwards along \(\typeof_{0}\).
\end{remark}

\begin{proof}[Proof of \cref{thm:tth-pi-ump}]
  The proof of this universal property is almost the same as that of
  \cref{thm:gat-ump}, but we add to \(\iL\obj \in \DTT_{\tth_{\Pi}}\)
  equations for \(\Pi\)-types:
  \begin{itemize}
  \item an equation
    \begin{multline*}
      (\var_{0} : \ty_{0}, \dots, \var_{\nat-1} :
      \ty_{\nat-1}(\var_{0}, \dots, \var_{\nat-2})) \To \\
      \textstyle{\prod_{\var_{\nat} : \ty_{\nat}(\var_{0}, \dots,
          \var_{\nat-1})}\ty(\var_{0}, \dots, \var_{\nat})} =
      (\poly_{\typeof}^{\nat}\PiI \circ \ty)(\var_{0}, \dots,
      \var_{\nat-1})
    \end{multline*}
    for any arrow \(\ty : \obj \to \poly_{\typeof}^{\nat+1}\Ty\);
  \item an equation
    \begin{multline*}
      (\var_{0} : \ty_{0}, \dots, \var_{\nat-1} :
      \ty_{\nat-1}(\var_{0}, \dots, \var_{\nat-2})) \To \\
      \lambda(\var_{\nat} : \ty_{\nat}(\var_{0}, \dots, \var_{\nat-1})).\el(\var_{0}, \dots, \var_{\nat}) =
      (\poly_{\typeof}^{\nat}\lambdaI \circ \el)(\var_{0}, \dots,
      \var_{\nat-1})
    \end{multline*}
    for any arrow \(\ty : \obj \to \poly_{\typeof}^{\nat+1}\Ty\) and
    any arrow \(\el : \obj \to \poly_{\typeof}^{\nat+1}\El\) over
    \(\ty\).
  \end{itemize}
\end{proof}
\endgroup

Universal properties for type theories with inductive types get
complicated. We only describe the simplest case.

\begin{example}
  Let \(\tth_{\Empty}\) be the type theory with the empty type
  \(\Empty\) which is the inductive type without constructors and with
  the following elimination rule.
  \[
    \inferrule
    {\ctx, \var : \Empty \vdash \tyI : \s{Type} \\
      \ctx \vdash \el : \Empty}
    {\ctx \vdash \s{elim}_{\Empty}(\var.\tyI, \el) : {\tyI}[\el/\var]}
  \]
  Then \(\thdtt_{\tth_{\Empty}}\) contains the arrow
  \(\Empty : \term \to \Ty_{0}\) represented by the interpretation
  \(\Ty_{0} \to \term\) defined by
  \begin{align*}
    \ty_{0}
    &\mapsto ({} \vdash \Empty : \s{Type})
  \end{align*}
  and the arrow \(\s{elim}_{\Empty} : \Empty^{*}\Ty_{1} \to \El_{1}\)
  represented by the interpretation \(\El_{1} \to \Empty^{*}\Ty_{1}\)
  defined by
  \begin{align*}
    \ty_{0}
    &\mapsto ({} \vdash \Empty : \s{Type}) \\
    \ty_{1}
    &\mapsto (\var_{0} : \Empty \vdash \ty_{1}(\var_{0}) : \s{Type})
  \\
    \el_{1}
    &\mapsto (\var_{0} : \Empty \vdash
      \s{elim}_{\Empty}(\var.\ty_{1}(\var), \var_{0}) :
      \ty_{1}(\var_{0})).
  \end{align*}
  These arrows make the diagram
  \begin{equation}
    \begin{tikzcd}
      & \El_{1}
      \arrow[d,"\typeof_{1}"] \\
      \Empty^{*}\Ty_{1}
      \arrow[ur,"\s{elim}_{\Empty}"]
      \arrow[r]
      \arrow[d] &
      \Ty_{1}
      \arrow[d,"\ctxof_{0}"] \\
      \term
      \arrow[r,"\Empty"'] &
      \Ty_{0}
    \end{tikzcd}
    \label[diagram]{eq:5}
  \end{equation}
  commute. \(\thdtt_{\tth_{\Empty}}\), together with \cref{eq:5}, is
  universal in the following sense.
\end{example}

\begin{theorem}
  \def\EmptyI{Z}
  \def\elimI{j}

  For any cartesian category \(\cart\), exponentiable arrow
  \(\typeof : \El \to \Ty\) and arrows \(\EmptyI : \term \to \Ty\) and
  \(\elimI : \EmptyI^{*}\poly_{\typeof}\Ty \to \poly_{\typeof}\El\) such
  that the diagram
  \begin{equation}
    \begin{tikzcd}
      & \poly_{\typeof}\El
      \arrow[d,"\poly_{\typeof}\typeof"] \\
      \EmptyI^{*}\poly_{\typeof}\Ty
      \arrow[ur,"\elimI"]
      \arrow[r]
      \arrow[d] &
      \poly_{\typeof}\Ty
      \arrow[d] \\
      \term
      \arrow[r,"\EmptyI"'] &
      \Ty
    \end{tikzcd}
    \label[diagram]{eq:6}
  \end{equation}
  commutes, there exists a unique, up to unique isomorphism, cartesian
  functor \(\functor : \thdtt_{\tth_{\Empty}} \to \cart\) such that
  \(\typeof \cong \functor\typeof_{0}\) and \(\functor\) sends
  pushforwards along \(\typeof_{0}\) to those along \(\typeof\) and
  \cref{eq:5} to \cref{eq:6}.
\end{theorem}
\begin{proof}
  Analogous to \cref{thm:tth-pi-ump}.
\end{proof}

\section*{Acknowledgement}
\label{sec:acknowledgement}

I gratefully acknowledge Benno van den Berg for helpful feedback and
corrections on drafts of this paper. I would also like to thank Thomas
Streicher for discussions. The referee made valuable suggestions which
greatly improved the paper. This work is part of the research
programme ``The Computational Content of Homotopy Type Theory'' with
project number 613.001.602, which is financed by the Netherlands
Organisation for Scientific Research (NWO).


\printbibliography

\end{document}